\documentclass[11pt]{article}

\usepackage[english]{babel}
\usepackage{amsmath,amsthm,amssymb,bm,thmtools,enumitem,mathrsfs,mathtools,titlesec,tikz,tikz-cd}
\usepackage[left=2.4cm,right=2.4cm,top=2.4cm,bottom=2.4cm,bindingoffset=0cm]{geometry}

\usepackage{quoting}
\quotingsetup{vskip=0pt}

\setlist[enumerate]{topsep=2pt,label=\textup{(\arabic*)},leftmargin=2em,labelsep=.5em}
\setlist{noitemsep}

\usepackage[T1]{fontenc}

\declaretheoremstyle[
  spaceabove=\topsep, spacebelow=2ex,
  headfont=\normalfont\bfseries,
  notefont=\mdseries, notebraces={(}{)},
  bodyfont=\normalfont\itshape,
  postheadspace=.5em,
  qed=\qedsymbol
]{mystyle}

\declaretheoremstyle[
  spaceabove=\topsep, spacebelow=6pt,
  headfont=\normalfont\bfseries,
  notefont=\mdseries, notebraces={(}{)},
  bodyfont=\normalfont,
  postheadspace=.5em,
  qed=\qedsymbol
]{mydefstyle}

\theoremstyle{mystyle}
\declaretheorem[numberlike=subsection]{proposition}
\declaretheorem[numberlike=subsection]{theorem}

\declaretheorem[numberlike=subsection]{lemma}

\theoremstyle{mydefstyle}
\declaretheorem[numberlike=subsection]{definition}

\numberwithin{equation}{subsection}

\titleformat{\section}[block]
  {\filcenter\normalfont\large\bfseries}{\thesection.}{.5em}{}
\titleformat{\subsection}[runin]
  {\normalfont\bfseries}{\thesubsection.}{.5em}{}
\titlespacing*{\section}{0pt}{4ex plus .2ex minus 1ex}{3ex plus .2ex minus 1ex}
\titlespacing*{\subsection}{0pt}{2ex plus .2ex minus 1ex}{.5em}
\newcommand{\bbA}{\mathbb{A}}
\newcommand{\bbC}{\mathbb{C}}
\newcommand{\bbG}{\mathbb{G}}
\newcommand{\bbH}{\mathbb{H}}
\newcommand{\bbN}{\mathbb{N}}

\newcommand{\bbQ}{\mathbb{Q}}
\newcommand{\bbR}{\mathbb{R}}
\newcommand{\bbS}{\mathbb{S}}
\newcommand{\bbZ}{\mathbb{Z}}

\newcommand{\Af}{\bbA_{\mathrm{f}}}

\newcommand{\cA}{\mathscr{A}}

\newcommand{\cC}{\mathscr{C}}

\newcommand{\cE}{\mathscr{E}}

\newcommand{\cG}{\mathscr{G}}

\newcommand{\cP}{\mathscr{P}}

\newcommand{\cY}{\mathscr{Y}}

\newcommand{\ff}{\mathfrak{f}}

\newcommand{\fH}{\mathfrak{H}}

\newcommand{\fS}{\mathfrak{S}}

\newcommand{\biI}{\bm{I}}

\newcommand{\Aut}{\mathrm{Aut}}
\newcommand{\Br}{\mathrm{Br}}
\newcommand{\Cores}{\mathrm{Cor}}
\newcommand{\CSp}{\mathrm{GSp}}
\newcommand{\Emb}{\mathrm{Emb}}
\newcommand{\End}{\mathrm{End}}
\newcommand{\Gal}{\mathrm{Gal}}
\newcommand{\GL}{\mathrm{GL}}

\newcommand{\GSp}{\mathrm{GSp}}
\newcommand{\GU}{\mathrm{GU}}
\newcommand{\Hom}{\mathrm{Hom}}
\newcommand{\Image}{\mathrm{Im}}

\newcommand{\MT}{\mathrm{MT}}
\newcommand{\Nm}{\mathrm{Nm}}
\newcommand{\PGL}{\mathrm{PGL}}

\newcommand{\Res}{\mathrm{Res}}
\newcommand{\Sh}{\mathrm{Sh}}
\newcommand{\SL}{\mathrm{SL}}

\newcommand{\St}{\mathrm{St}}
\newcommand{\Stab}{\mathrm{Stab}}
\newcommand{\SU}{\mathrm{SU}}
\newcommand{\Tr}{\mathrm{Tr}}

\newcommand{\UU}{\mathrm{U}}

\newcommand{\ad}{\mathrm{ad}}
\newcommand{\can}{\mathrm{can}}

\newcommand{\cor}{\mathrm{cor}}

\newcommand{\der}{\mathrm{der}}
\newcommand{\diag}{\mathrm{diag}}
\newcommand{\id}{\mathrm{id}}
\newcommand{\inv}{\mathrm{inv}}
\newcommand{\mult}{\mathrm{m}}
\newcommand{\nc}{\mathrm{nc}}
\newcommand{\op}{\mathrm{op}}

\newcommand{\pr}{\mathrm{pr}}

\newcommand{\res}{\mathrm{res}}
\newcommand{\restr}{\mathrm{restr}}
\newcommand{\sconn}{\mathrm{sc}}
\newcommand{\sym}{\mathrm{s}}
\newcommand{\wt}{\mathsf{w}}

\newcommand{\tto}{\longrightarrow}
\newcommand{\isomarrow}{\xrightarrow{\,\sim\, }}

\newcommand{\Qbar}{\overline{\bbQ}}

\newcommand{\kbar}{\bar{k}}

\let\epsilon=\varepsilon
\let\phi=\varphi

\begin{document}

\begin{center}
\textbf{\LARGE A classification of Shimura curves in $\cA_g$}
\bigskip

\textit{by}
\bigskip

{\Large Ben Moonen}
\end{center}
\vspace{8mm}

{\small %% for abstract and MSC

\noindent
\begin{quoting}
\textbf{Abstract.} We give a precise classification, in terms of Shimura data, of all $1$-dimensional Shimura subvarieties of a moduli space of polarized abelian varieties.
\medskip

\end{quoting}

} %% end of small
\vspace{4mm}

\section{Introduction}

\subsection{}
The goal of this note is to give a classification of the $1$-dimensional Shimura subvarieties of~$\cA_g$, the moduli space of $g$-dimensional (polarized) abelian varieties. My motivation to write this up comes partly from the fact that at several places in the literature there seem to be misconceptions about this. (See~\ref{rem:ViehZuo} for an example.)

In terms of Shimura data, what we want to classify are triples $(G,\cY,\rho)$, where $(G,\cY)$ is a Shimura datum such that $\cY$ is a $1$-dimensional complex manifold, and $\rho \colon (G,\cY) \hookrightarrow (\CSp_{2g},\fH_g^{\pm})$ is an embedding into a Siegel modular Shimura datum. The adjoint Shimura datum $(G^\ad,\cY^\ad)$ is then easy to describe: Take a $4$-dimensional central simple algebra~$D$ over a totally real field~$F$ which splits at precisely one of the real places of~$F$, and let $\cG_D = \PGL_{1,D}$. There is a unique $\cG_D(\bbR)$-conjugacy class~$\cY_D$ of homomorphisms $\bbS \to \cG_{D,\bbR}$ such that $(\cG_D,\cY_D)$ is a Shimura datum (see Section~\ref{subsec:GadDad}), and for every triple $(G,\cY,\rho)$ as above, $(G^\ad,\cY^\ad)$ is isomorphic to a datum of this form~$(\cG_D,\cY_D)$.

The problem at hand can be reduced to the situation where the generic abelian variety over the Shimura curve given by $(G,\cY)$ is simple. If the endomorphism algebra is of Albert type I, II or~III, which means that its centre is a totally real field, knowing the adjoint Shimura datum essentially solves the whole problem, as the connected centre of~$G$ then equals~$\bbG_{\mult}$ and the representation~$\rho$ can only be the so-called corestriction representation. Typical examples of what we obtain are the $1$-dimensional Shimura subvarieties of~$\cA_4$ constructed by Mumford~\cite{MumfordNote},~\S~4; what seems less well-known is that there is also a quaternionic version of Mumford's construction that gives rise to abelian varieties of Albert types~II and~III.

The most interesting part of the problem is the case where the generic abelian variety is of Albert type~IV, so that the centre of the endomorphism algebra is a CM field. In this case, if we fix $F$ and~$D$ as above such that $(G^\ad,\cY^\ad) \cong (\cG_D,\cY_D)$, the derived group of~$G$ is a finite cover of~$\cG_D$, and the possibilities for $\rho|_{G^\der}$ correspond to the $\Gal(\Qbar/\bbQ)$-orbits of nonempty subsets of the set $\Emb(F) = \Hom(F,\Qbar)$. The main point is that the representation of~$G^\der$ needs to be combined with a nontrivial representation of the centre of~$G$ in order to obtain a Shimura datum that embeds into a Siegel modular datum. We carry out a precise analysis of the data involved. For the final result we refer to Propositions~\ref{prop:Albert4} and~\ref{prop:PhiProp}, and then Theorem~\ref{thm:MainThm} states that the description we have found covers all cases. An interesting feature is that we are naturally led to introduce a notion of a `partial CM type', and that our classification involves a condition that generalizes the classical notion of primitivity of a CM type. (The notion we consider is more specific than the one used in~\cite{PST}.)

\subsection{}
In Sections~\ref{sec:RepTh} and~\ref{sec:Exa} we review the results from representation theory that we need, following Tits's paper~\cite{Tits}, and we discuss two examples of representations that play a key role. In Section~\ref{sec:AV1dimSV} we study simple complex abelian varieties~$X$ whose associated Shimura datum is $1$-dimensional. This is the main part of the paper. As mentioned, the most interesting case to analyse is when $X$ is of Albert type~IV. In Section~\ref{sec:Classif} we explain how the analysis carried out in Section~\ref{sec:AV1dimSV} gives a complete solution to the classification problem in the case when the generic abelian variety is simple, and in Section~\ref{sec:Nonsimple} we extend this to the general case.

\subsection{}
There are several papers that discuss the classification of Shimura (sub)varieties, and one may wonder to what extent the results in the present paper are already covered in the literature. To my knowledge, the precise classification carried out here is new, though I suspect, based on Remarque~2.3.11 in~\cite{DelCorvalis}, that the results have been long known to Deligne. The work of Satake~\cite{Satake} and the subsequent work of Addington~\cite{Addington} does discuss the representation theory involved but does not contain the results that we obtain. (These papers focus on the representation theory of the semisimple part of the Mumford--Tate group, whereas in our work it is the interplay between the representation theory of the semisimple part and the centre that plays a main role. The `chemistry terminology' of~\cite{Addington} is not commonly used; we use root data instead.)

\subsection{Notation and conventions.}\label{subsec:Notat}
\begin{enumerate}
\item Throughout, $\Qbar$ is viewed as a subfield of~$\bbC$ and we write $\Gamma_\bbQ = \Gal(\Qbar/\bbQ)$. If $K$ is a number field, we write $\Emb(K) = \Hom(K,\Qbar)$, which is identified with the set of complex embeddings of~$K$ or, in case $K$ is totally real, the set of real embeddings of~$K$.

\item If $G$ is a reductive group, we denote by $G^\der$ its derived group, by $G^\ad$ the adjoint group, and by~$G^\sconn$ the simply connected cover of~$G^\ad$.

\item If $X$ is an abelian variety, we denote by $\End^0(X) = \End(X) \otimes \bbQ$ its endomorphism algebra.
\end{enumerate}

\section{Some results from representation theory}\label{sec:RepTh}

\noindent
Throughout this section, $k$ denotes a field of characteristic~$0$ with algebraic closure $k \subset \kbar$.

\subsection{Basic notions.}\label{subsec:RepBasic} If $G$ is an algebraic group over~$k$ then by a representation~$\rho$ of~$G$ we mean a representation on a finite dimensional $k$-vector space. If we say that $\rho$ is irreducible, we mean it is irreducible over~$k$. If $k \subset L$ is a field extension, we denote by~$\rho_L$ the representation of~$G_L$ obtained by extension of scalars. 

If $D$ is a $k$-algebra and $V$ is a right $D$-module of finite $k$-dimension, we denote by~$\GL_D(V)$ the algebraic group over~$k$ of $D$-linear automorphisms of~$V$, and by $\restr_{D/k} \colon \GL_D(V) \to \GL_k(V)$ the canonical homomorphism. (Instead of~$\GL_k(V)$ also the notation $\GL({}_k V)$ is used, where ${}_k V$ denotes the underlying $k$-vector space of~$V$.) If $G \to \GL_D(V)$ is a homomorphism, we refer to~$V$ as a $D$-$G$-module, and we denote by $\End_{D\text{-}G}(V)$ the algebra of $D$-linear endomorphisms of~$V$ that are $G$-equivariant. 

For $D$ a central simple $k$-algebra, we define $\GL_{n,D} = \GL_D(D^n)$, where $D^n$ is viewed as a right $D$-module. We write $\SL_{1,D} \subset \GL_{1,D}$ for the kernel of the norm homomorphism $\GL_{1,D} \to \bbG_{\mult,k}$ and $\PGL_{1,D}$ for the cokernel of $\bbG_{\mult,k} \to \GL_{1,D}$.

If $k \subset L$ is a finite field extension and $R\colon G_L \to \GL(V)$ is a representation of~$G_L$ over~$L$, we denote by $\res_{L/k}(R)$ the representation of~$G$ (over~$k$) given by the composition
\[
\res_{L/k}(R)\colon G \xrightarrow{~\can~} \Res_{L/k}(G_L) \xrightarrow{~\Res_{L/k}(R)~} \Res_{L/k}\bigl(\GL(V)\bigr) \xhookrightarrow{~\restr_{L/k}~} \GL_k(V)\, .
\]

\subsection{Representations of tori.}\label{subsec:reptori} If $T/k$ is a torus, let $\mathsf{X}^*(T) = \Hom(T_{\kbar},\bbG_{\mult,\kbar})$ be the character group of~$T$ and $\mathsf{X}_*(T) = \Hom(\bbG_{\mult,\kbar},T_{\kbar})$ the cocharacter group. These are free $\bbZ$-modules of finite rank equipped with a continuous action of~$\Gal(\kbar/k)$. We have a Galois-equivariant perfect pairing $\mathsf{X}^*(T) \times \mathsf{X}_*(T) \to \bbZ$. For $\xi \in \mathsf{X}^*(T)$, write~$\kbar_\xi$ for the vector space~$\kbar$ on which $T_{\kbar}$ acts through the character~$\xi$, and let $1_\xi \in \kbar_\xi$ be the identity element. 

The irreducible representations of~$T$ correspond to the $\Gal(\kbar/k)$-orbits in~$\mathsf{X}^*(T)$. If $\Xi \subset \mathsf{X}^*(T)$ is such an orbit, the corresponding representation~$\rho_\Xi$ can be constructed by considering the $\kbar$-vector space $V_{\Xi,\kbar} = \oplus_{\xi\in\Xi}\; \kbar_\xi$, on which $\Gal(\kbar/k)$ acts by the rule
\[
\gamma\cdot \Bigl(\sum_{\xi \in \Xi}\; c_\xi \cdot 1_\xi\Bigr) = \sum_{\xi \in \Xi}\; \gamma(c_{\gamma^{-1} \cdot \xi}) \cdot 1_\xi\qquad \text{(for $\gamma \in \Gal(\kbar/k)$ and coefficients $c_\xi \in \kbar$).}
\]
(In particular, $\gamma$ sends $c\cdot 1_\xi$ to $\gamma(c) \cdot 1_{\gamma \cdot \xi}$.) The representation of~$T_{\kbar}$ on~$V_{\Xi,\kbar}$ descends to a representation of~$T$ on the $k$-vector space $V_\Xi = (V_{\Xi,\kbar})^{\Gal(\kbar/k)}$, and this gives the representation~$\rho_\Xi$. By construction, $\rho_{\Xi,\kbar} \cong V_{\Xi,\kbar}$ as representations of~$T_{\kbar}$. 

For $\xi \in \Xi$, define $k(\xi) = \kbar^{\Stab(\xi)}$, where $\Stab(\xi) \subset \Gal(\kbar/k)$ is the stabilizer of~$\xi$. The choice of an element $\xi_0 \in \Xi$ gives an isomorphism $k(\xi_0) \isomarrow \End(\rho_\Xi)$. Concretely, if $y \in k(\xi_0)$ and $\xi \in \Xi$, choose $\gamma \in \Gal(\kbar/k)$ such that $\xi = \gamma\cdot \xi_0$ and let $y$ act on~$\kbar_\xi$ as multiplication by~$\gamma(y)$, which is independent of the choice of~$\gamma$. Moreover, one readily checks that the action of~$k(\xi_0)$ on~$V_{\Xi,\kbar}$ thus obtained commutes with the action of $\Gal(\kbar/k)$ and is $T_{\kbar}$-equivariant; hence it descends to a homomorphism of $k$-algebras $k(\xi_0) \to \End(\rho_\Xi)$. To see that this is an isomorphism, it suffices to note that both sides have the same $k$-dimension because $k(\xi_0) \otimes_k \kbar \cong \prod_{\xi \in \Xi}\; \kbar \cong \End(\rho_{\Xi,\kbar})$.

\subsection{Example.}\label{exa:TK}
If $E$ is a number field, let $T_E = \Res_{E/\bbQ}\; \bbG_{\mult,E}$, which is a torus over~$\bbQ$ of rank equal to $[E:\bbQ]$. The character group is given by $\mathsf{X}^*(T_E) = \bigoplus_{\phi \in \Emb(E)}\; \bbZ\cdot \mathrm{e}_\phi$, where $\mathrm{e}_\phi$ denotes the character induced by the embedding~$\phi$. The Galois group $\Gamma_\bbQ = \Gal(\Qbar/\bbQ)$ acts on~$\mathsf{X}^*(T_E)$ through its action on~$\Emb(E)$. The cocharacter group is $\mathsf{X}_*(T_E) = \oplus_{\phi \in \Emb(E)}\; \bbZ\cdot \check{\mathrm{e}}_\phi$, where $\{\check{\mathrm{e}}_\phi\}_{\phi \in \Emb(E)}$ is the dual basis.

The set of elements $\{\mathrm{e}_\phi\}_{\phi \in \Emb(E)}$ is a $\Gamma_\bbQ$-orbit in~$\mathsf{X}^*(T_E)$. We denote the corresponding irreducible representation by~$\St_E$, and we refer to it as the standard representation of~$T_E$. It is given by the canonical homomorphism $T_E = \Res_{E/\bbQ}\; \GL_{1,E} \tto \GL({}_\bbQ E)$, where ${}_\bbQ E$ denotes the $\bbQ$-vector space underlying~$E$. The endomorphism algebra of~$\St_E$ is~$E$.

\subsection{Review of some results of Tits.}\label{subsec:Tits} We briefly review some results by Tits~\cite{Tits}. (What we have discussed in~\ref{subsec:reptori} is a very special case of this.)

Let $G/k$ be a reductive group. This group gives rise, in a canonical way, to a based root datum $(\mathsf{X},\Phi,\Delta,\mathsf{X}^\vee,\Phi^\vee,\Delta^\vee)$ with an action of $\Gal(\kbar/k)$; see for instance \cite{Conrad}, especially Remark~7.1.2. As in the case of a torus, $\mathsf{X}$ and~$\mathsf{X}^\vee$ are free $\bbZ$-modules with $\Gal(\kbar/k)$-action and we have a Galois-equivariant perfect pairing $\langle~,~\rangle\colon \mathsf{X} \times \mathsf{X}^\vee \to \bbZ$. The sub-lattices $\bbZ \cdot \Phi \subset \mathsf{X}$ and $\bbZ \cdot \Phi^\vee \subset \mathsf{X}^\vee$ are called the root lattice and the co-root lattice. Define 
\[
\mathsf{X}_+ = \bigl\{\xi \in \mathsf{X} \bigm| \langle\xi,\varpi^\vee\rangle \geq 0\quad \text{for all $\varpi^\vee \in \Phi^\vee$}\bigr\}\, .
\]
For $\xi \in \mathsf{X}_+$, let $\rho_{\kbar,\xi}$ denote the irreducible representation of~$G_{\kbar}$ with highest weight~$\xi$.

Let $\mathsf{X}_0 \subset \mathsf{X}$ be the subgroup that is generated by~$\Phi$ and by the elements that are perpendicular to~$\Phi^\vee$. Following~\cite{Tits}, define 
\[
\cC^*(G) = \mathsf{X}/\mathsf{X}_0\, ,
\] 
which is a finite group that only depends on~$G^\der$. It comes equipped with an action of $\Gal(\kbar/k)$.

If $\xi \in \mathsf{X}_+$ is invariant under $\Gal(\kbar/k)$, there exists a division algebra~$D = D_\xi$ with centre~$k$ and a representation $r_\xi \colon G \to \GL_D(V)$, for some right $D$-module~$V$ of finite type, such that:
\begin{itemize}
\item the representation $\rho_\xi = \restr_{D/k} \circ r_\xi \colon G \to \GL(V)$ is irreducible (notation as in \ref{subsec:RepBasic});

\item if $d = \deg(D)$ is the degree of~$D$ (i.e., $\dim_k(D) = d^2$), the representation $\rho_{\xi,\kbar}$ is isomorphic to the sum of $d$ copies of~$\rho_{\kbar,\xi}$.
\end{itemize}
(Note that $\rho_{\xi,\kbar} = (\rho_\xi)_{\kbar}$ is not the same as $\rho_{\kbar,\xi}$.) The division algebra~$D_\xi$ is unique up to isomorphism, and given~$D_\xi$ the representation~$r_\xi$ is unique up to $D_\xi$-equivalence. (See \cite{Tits}, Th\'eor\`eme~3.3.) If the context requires it, we write $\rho_{k,\xi}$ to indicate the ground field.

Let $\Br(k)$ be the Brauer group of~$k$. There exists a homomorphism
\[
\beta_{G,k} \colon \cC^*(G)^{\Gal(\kbar/k)} \to \Br(k)
\]
that only depends on~$G^\der$, with the property that for a Galois-invariant dominant weight $\xi \in (\mathsf{X}_+)^{\Gal(\kbar/k)}$ as above, $\beta_{G,k}(\xi \bmod \mathsf{X}_0) = [D_\xi]$.

With this notation, the general description of the irreducible representations of~$G$ is as follows. For $\xi \in \mathsf{X}_+$, let $k(\xi) \subset \kbar$ be the field extension of~$k$ that corresponds to the stabilizer of~$\xi$ in~$\Gal(\kbar/k)$. By what we have just explained, there exists a division algebra $D = D_\xi$ with centre~$k(\xi)$ and a representation $r_{k(\xi),\xi} \colon G_{k(\xi)} \to \GL_D(V)$ over~$k(\xi)$ such that $\rho_{k(\xi),\xi} = \restr_{D/k(\xi)} \circ r_{k(\xi),\xi} \colon G_{k(\xi)} \to \GL_{k(\xi)}(V)$ is an irreducible representation of~$G_{k(\xi)}$ which after extension of scalars to~$\kbar$ becomes a sum of copies of~$\rho_{\kbar,\xi}$. Then
\[
\rho_\xi = \res_{k(\xi)/k}\bigl(\rho_{k(\xi),\xi}\bigr) \colon G \tto \GL_k(V)
\] 
(notation as in \ref{subsec:RepBasic}) is an irreducible representation of~$G$. If necessary we write~$\rho_{k,\xi}$ instead of~$\rho_\xi$ to indicate the ground field, and again we note that if $k \subset L$ is a field extension, $\rho_{L,\xi}$ is in general not the same as~$\rho_{\xi,L}$, the extension of scalars of $\rho_\xi$ to~$L$. The isomorphism class of the representation~$\rho_\xi$ only depends on the $\Gal(\kbar/k)$-orbit of~$\xi$, and every irreducible representation of~$G$ is of the form~$\rho_\xi$ for some $\xi \in \mathsf{X}_+$. If $d = \deg(D_\xi)$ and $\Gal(\kbar/k) \cdot \xi = \{\xi_1,\ldots,\xi_r\}$ then
\[
\bigl(\rho_\xi\bigr)_{\kbar} \cong \bigl(\rho_{\kbar,\xi_1}\bigr)^{\oplus d} \oplus \cdots \oplus \bigl(\rho_{\kbar,\xi_r}\bigr)^{\oplus d}\, .
\]
The endomorphism algebra of~$\rho_\xi$ is isomorphic to~$D_\xi^\op$. (See the proof of Th\'eor\`eme~7.2 in~\cite{Tits}.)

\section{Examples}\label{sec:Exa}

We discuss two examples that play an important role in the next sections. As before, $k$ is a field with $\mathrm{char}(k) = 0$.

\subsection{Example.}\label{exa:111Rep}
Let $L = L_1 \times \cdots \times L_s$ be a product of finite field extensions of~$k$. Let $D = D_1 \times \cdots \times D_s$, where $D_j$ is a $4$-dimensional central simple $L_j$-algebra ($j=1,\ldots,s$). With notation as in~\ref{subsec:RepBasic}, let $\cG_j = \Res_{L_j/k}\; \SL_{1,D_j}$, and take $\cG = \Res_{L/k}\; \SL_{1,D} = \cG_1 \times \cdots \times \cG_s$. We have
\begin{equation}\label{eq:Gkbar}
\cG_{\kbar} \cong \prod_{\sigma \in \Hom_k(L,\kbar)}\; \SL_{2,\kbar}\, .
\end{equation}
The weight lattice of~$\cG$ is given by $\mathsf{X} = \bigoplus_{\sigma \in \Hom_k(L,\kbar)}\; \bbZ$, on which $\Gal(\kbar/k)$ acts through its action on~$\Hom_k(L,\kbar)$. We normalise this in such a manner that a weight $\xi = (\xi_\sigma)_{\sigma \in \Hom_k(L,\kbar)}$ is dominant if and only if $\xi_\sigma \geq 0$ for all~$\sigma$. Let $\rho_{\cor}$ be the irreducible representation of~$\cG$ corresponding to the weight $\xi$ with $\xi_\sigma = 1$ for all $\sigma \in \Hom_k(L,\kbar)$. We call~$\rho_{\cor}$ the corestriction representation of~$\cG$ over~$k$; it can be described as follows.

For $\sigma \in \Hom_k(L,\kbar)$, write $D_\sigma = D \otimes_{L,\sigma} \kbar$, which is isomorphic to the matrix algebra~$M_2(\kbar)$. On the ring $\otimes_{\sigma \in \Hom_k(L,\kbar)}\; D_\sigma$ (tensor product over~$\kbar$) we have a natural action of $\Gal(\kbar/k)$, which extends the action on~$\kbar$ (which is the centre). The corestriction of~$D$, notation $\Cores_{L/k}\; D$, is defined as the $k$-algebra of Galois-invariants:
\[
\Cores_{L/k}\; D = \Bigl( \bigotimes_{\sigma \in \Hom_k(L,\kbar)}\; D_\sigma\Bigr)^{\Gal(\kbar/k)}\, ,
\]
which is a central simple $k$-algebra. (See \cite{Tits}, Section~5.3, or \cite{Riehm}; for a more intrinsic approach, see~\cite{Ferrand}.) If $q = \dim_k(L) = \sum_{j=1}^s\; [L_j:k]$ then $\Cores_{L/k}\; D$ has degree~$2^q$ over~$k$. The canonical homomorphism $\Cores_{L/k}\; D \otimes_k \kbar \to \otimes_{\sigma \in \Hom_k(L,\kbar)}\; D_\sigma$ is an isomorphism of $\kbar$-algebras. On Brauer groups, the class $\bigl[\Cores_{L/k}\; D\bigr] \in \Br(k)$ is the image of $[D] \in \Br(L)$ under the corestriction map in Galois cohomology.

Writing $C = \Cores_{L/k}\; D$, we have a homomorphism $\alpha\colon \Res_{L/k}\; \GL_{1,D} \to \GL_{1,C}$, which on $\kbar$-valued points is given by the natural homomorphism 
\[
\prod_{\sigma \in \Hom_k(L,\kbar)}\; D_\sigma^* \to \Bigl(\bigotimes_{\sigma \in \Hom_k(L,\kbar)}\; D_\sigma\Bigr)^*\, .
\]
Let $W$ be the unique (up to isomorphism) simple left $C$-module, viewed as a $k$-vector space, so that we have a representation $\GL_{1,C} \to \GL(W)$. Then the corestriction representation is given by the composition
\[
\rho_{\cor}\colon \cG \hookrightarrow \Res_{L/k}\; \GL_{1,D} \xrightarrow{~\alpha~} \GL_{1,C} \to \GL(W)\, .
\]

In more detail, let $\cE = \End_C(W) = \End(\rho_{\cor})$ be the division algebra with centre~$k$ that is Brauer equivalent to~$C^\op$. (For the identity $\cE =  \End(\rho_{\cor})$, cf.\ the end of Section~\ref{subsec:Tits}.) Either $\cE = k$ or $\cE$ is a quaternion algebra over~$k$. If $\cE = k$ then $C \cong M_{2^q}(k)$; so we find that $W = k^{2^q}$ and $\GL_{1,C} \cong \GL_{2^q,k}$, and then $\rho_{\cor}\colon \cG \to \GL(W)$ is given by the homomorphism~$\alpha$. In this case, $\rho_{\cor,\kbar} \cong \boxtimes_{\sigma \in \Hom_k(L,\kbar)}\; \St_\sigma$, where by $\St_\sigma$ we mean the irreducible $2$-dimensional representation of the factor $\SL_{2,\kbar}$ in \eqref{eq:Gkbar} indexed by~$\sigma$. If $\cE$ is a quaternion algebra over~$k$ then $C \cong M_{2^{q-1}}(\cE^\op)$. Fixing such an isomorphism, we obtain $W = (\cE^\op)^{\oplus 2^{q-1}} \cong k^{2^{q+1}}$ and $\GL_{1,C} \cong \GL_{2^{q-1},\cE^\op}$. In this case, $\rho_{\cor}$ is the composition
\[
\cG \xrightarrow{~\alpha~} \GL_{2^{q-1},\cE^\op} \xrightarrow{~\restr_{\cE^\op/k}~} \GL_{2^{q+1},k}\, ,
\]
and $\rho_{\cor,\kbar}$ is isomorphic to a sum of two copies of $\boxtimes_{\sigma \in \Hom_k(L,\kbar)}\; \St_\sigma$.

\subsection{Example.}\label{exa:rhoIrep}
Let $F$ be a number field and $D$ a $4$-dimensional central simple algebra over~$F$. Consider the algebraic group $\cG = \Res_{F/\bbQ}\; \SL_{1,D}$ over~$\bbQ$, which is a simply connected semisimple group. The weight lattice of~$\cG$ is given by $\mathsf{X} = \oplus_{\sigma \in \Emb(F)}\; \bbZ$, on which $\Gamma_\bbQ = \Gal(\Qbar/\bbQ)$ acts through its action on~$\Emb(F)$. Again we normalise this such that a weight $\xi = (\xi_\sigma)_{\sigma \in \Emb(F)}$ is dominant if and only if $\xi_\sigma \geq 0$ for all~$\sigma$.

If $I \subset \Emb(F)$ is a subset, let $\xi_I$ be the weight given by the rule that $\xi_\sigma = 1$ if $\sigma \in I$ and $\xi_\sigma = 0$ otherwise. Let $\rho_I$ be the irreducible representation of~$\cG$ that corresponds to the $\Gamma_\bbQ$-orbit of~$\xi_I$. (In particular, $\rho_I \cong \rho_{\gamma(I)}$ for $\gamma \in \Gamma_\bbQ$.) This representation can be described as follows.

Let $\tilde{F} \subset \Qbar$ be the normal closure of~$F$. Writing $D_\sigma = D \otimes_{F,\sigma} \tilde{F}$, we have $\cG_{\tilde{F}} \cong \prod_{\sigma \in \Emb(F)}\; \SL_{1,D_\sigma}$, and hence $\cG_{\Qbar} \cong \prod_{\sigma \in \Emb(F)}\; \SL_{2,\Qbar}$. 

Let $k_I \subset \tilde{F}$ be the subfield of elements that are invariant under $\Stab(I) = \bigl\{\gamma \in \Gal(\tilde{F}/\bbQ) \bigm| \gamma(I) = I\bigr\}$. (This field~$k_I$ takes the role of what in~\ref{subsec:Tits} was called~$k(\xi)$.) The $k_I$-algebra $F \otimes_\bbQ k_I$ is a product of field extensions of~$k_I$. We have a natural isomorphism
\[
F \otimes_\bbQ k_I = \bigl(F \otimes_\bbQ \tilde{F})^{\Stab(I)} \cong \Bigl(\prod_{\sigma \in \Emb(F)}\; \tilde{F}\Bigr)^{\Stab(I)} = \Bigl(\prod_{\sigma \in I}\; \tilde{F} \times \prod_{\sigma \notin I}\; \tilde{F}\Bigr)^{\Stab(I)}\, .
\]
Defining 
\[
L_I = \Bigl(\prod_{\sigma \in I}\; \tilde{F}\Bigr)^{\Stab(I)}\, ,\qquad
L^\prime_I = \Bigl(\prod_{\sigma \notin I}\; \tilde{F}\Bigr)^{\Stab(I)}
\]
we get a decomposition $F \otimes_\bbQ k_I = L_I \times L_I^\prime$. Define $D_I = D \otimes_F L_I$ and $D^\prime_I = D \otimes_F L^\prime_I$. Then
\[
\cG_{k_I} = \bigl(\Res_{L_I/k_I}\; \SL_{1,D_I}\bigr) \times \bigl(\Res_{L^\prime_I/k_I}\; \SL_{1,D^\prime_I}\bigr)\, .
\]
Let $\rho_{\cor} \colon \Res_{L_I/k_I}\; \SL_{1,D_I} \to \GL(W)$ be the corestriction representation over~$k_I$ as in Example~\ref{exa:111Rep}, applied with $L = L_I$ and $D = D_I$. The representation~$\rho_I$ is then the composition
\[
\cG \xrightarrow{\can} \Res_{k_I/\bbQ}(\cG_{k_I}) \xrightarrow{~\pr~} \Res_{k_I/\bbQ}\bigl(\Res_{L_I/k_I}\; \SL_{1,D_I}\bigr) \xrightarrow{~\Res(\rho_{\cor})~} \Res_{k_I/\bbQ}\; \GL(W) \xrightarrow{~\restr_{k_I/\bbQ}~} \GL({}_\bbQ W)\, .
\]
(In other words, $\rho_I = \res_{k_I/\bbQ}(\rho_{\cor})$, where we view~$\rho_{\cor}$ as a representation of~$\cG_{k_I}$.)

The endomorphism algebra $\cE_I = \End(\rho_I)$ is a division algebra with centre~$k_I$ which is Brauer equivalent to $\Cores_{L_I/k_I}\; D_I$. Either $\cE_I = k_I$ or $\cE_I$ is a quaternion algebra over~$k_I$. The representation~$\rho_{I,\Qbar}$ is isomorphic to
\begin{equation}\label{eq:sumJ}
\bigoplus_{J \in \Gamma_\bbQ\cdot I}\; \left(\mathop{\boxtimes}\limits_{\sigma \in J}\; \St_\sigma\right)^{\oplus \deg(\cE_I)}\, ,
\end{equation}
where $\St_\sigma$ denotes the $2$-dimensional irreducible representation of~$\cG_{\Qbar}$ given by the standard irreducible representation of the factor indexed by~$\sigma$. For later use we note that, because $\Gal(\Qbar/\tilde{F})$ acts trivially on~$\mathsf{X}$, the summands that appear here are already defined over~$\tilde{F}$; more precisely: $\rho_{I,\tilde{F}}$ decomposes as a direct sum of representations~$R_{\tilde{F},J}$, for $J \in \Gamma_\bbQ\cdot I$, such that $\bigl(R_{\tilde{F},J}\big)_{\Qbar}$ is isomorphic to a sum of $\deg(\cE_I)$ copies of $\boxtimes_{\sigma \in J}\; \St_\sigma$.

\subsection{Remark.}\label{rem:rhoI} In the above description of the representation~$\rho_I$, we have broken the symmetry by choosing a representative~$I$ for the Galois-orbit $\Gamma_\bbQ \cdot I$. For what follows, it is important to restore the symmetry. We shall use the symbol~$\biI$ for a $\Gamma_\bbQ$-orbit of nonempty subsets of~$\Emb(F)$, and we write $\rho_{\biI} \colon \cG \to \GL(W_{\biI})$ for the corresponding irreducible representation of~$\cG = \Res_{F/\bbQ}\; \SL_{1,D}$ over~$\bbQ$ (as in~\ref{exa:rhoIrep}). Define $\cE_{\biI} = \End_{\cG}(\rho_{\biI})$, and let $k_{\biI}$ be the centre of~$\cE_{\biI}$. Then $\Emb(k_{\biI})$ is in natural bijection with~$\biI$, in such a way that if $I \in \biI$ corresponds to the embedding~$\tau$, the subfield $\tau(k_{\biI}) \subset \Qbar$ is the field~$k_I$ of Example~\ref{exa:rhoIrep} and $\cE_{\biI} \otimes_{k_{\biI},\tau} k_I \isomarrow \cE_I$. The normal closure of~$k_{\biI}$ is a subfield of~$\tilde{F}$. We denote by~$\ell(\biI)$ the cardinality of the sets $I \in \biI$, so that $\dim_\bbQ(W_{\biI}) = [k_{\biI}:\bbQ] \cdot \deg(\cE_{\biI}) \cdot 2^{\ell(\biI)}$.

\begin{lemma}\label{lem:Eisquat}
Let $F$, $D$ and~$\cG = \Res_{F/\bbQ}\; \SL_{1,D}$ be as in Example~\ref{exa:rhoIrep}. Let $\biI$ be a $\Gamma_\bbQ$-orbit of subsets of~$\Emb(F)$, and let $\rho_{\biI} \colon \cG \to \GL(W_{\biI})$ be the corresponding irreducible representation of~$\cG$ over~$\bbQ$. (Notation as in the previous remark.) Assume the following two conditions are satisfied:
\begin{enumerate}[label=\textup{(\alph*)}]
\item there is a unique embedding $\sigma \in \Emb(F)$ such that $D \otimes_{F,\sigma} \bbR$ is isomorphic to~$M_2(\bbR)$;

\item the sets $I \in \biI$ are nonempty and $\biI \neq \bigl\{\Emb(F)\bigr\}$.
\end{enumerate} 
Then the endomorphism algebra~$\cE_{\biI}$ of~$\rho_{\biI}$ is a quaternion algebra over its centre~$k_{\biI}$.
\end{lemma}

\begin{proof}
As explained above, we have a bijection $\biI \isomarrow \Emb(k_{\biI})$, and if $I \mapsto\tau$ then $\tau(k_{\biI})$ is the field~$k_I$ as in Example~\ref{exa:rhoIrep}, which is a subfield of~$\bbR$. With notation as in that example, the image in~$\Br(\bbR)$ of the class $[\cE_{\biI}] \in \Br(k_{\biI})$ under~$\tau$ is the class of
\[
\Cores_{(L_I \otimes_{k_I} \bbR)/\bbR}\bigl(D_I \otimes_{k_I} \bbR\bigr) = \bigotimes_{\sigma \in I}\; D_\sigma\, ,
\]
where $D_\sigma = D \otimes_{F,\sigma} \bbR$. This class is the sum over the elements $\sigma \in I$ of the classes $[D_\sigma] \in \Br(\bbR)$. Let $\sigma_\nc \in \Emb(F)$ be the unique real embedding at which~$D$ splits. Assumption~(b) implies that we can find $I_1$, $I_2 \in \biI$ such that $\sigma_{\nc} \in I_1$ and $\sigma_{\nc} \notin I_2$. The corresponding two classes in~$\Br(\bbR)$ are unequal, so the class $[\cE_{\biI}] \in \Br(k_{\biI})$  cannot be trivial.
\end{proof}

\section{Abelian varieties whose associated Shimura datum is $1$-dimensional}\label{sec:AV1dimSV}

\subsection{Notation related to Hodge structures.} As usual in Hodge theory, we define $\bbS = \Res_{\bbC/\bbR}\; \bbG_\mult$. The character group of this torus is given by $\mathsf{X}^*(\bbS) = \bbZ \oplus \bbZ$, with complex conjugation acting by $(p,q) \mapsto (q,p)$. The norm homomorphism $\Nm\colon \bbS \to \bbG_{\mult,\bbR}$ (on $\bbR$-points: $z \mapsto z\bar{z}$) corresponds to the character $(1,1)$. Define $\wt \colon \bbG_{\mult,\bbR} \to \bbS$ (on $\bbR$-points: the inclusion $\bbR^* \hookrightarrow \bbC^*$) to be the unique homomorphism such that $\Nm \circ \wt$ is $z \mapsto z^2$, and let $i \colon \bbG_{\mult,\bbC} \hookrightarrow \bbS_\bbC$ be the morphism given on $\bbC$-valued points by $z\mapsto (z,1)$; in terms of the natural pairing between characters and cocharacters, $i$ is described by its property that $\bigl\langle (1,0),i\bigr\rangle = 1$ and $\bigl\langle (0,1),i\bigr\rangle = 0$. 

A $\bbQ$-Hodge structure of weight~$n$ is given by a finite dimensional $\bbQ$-vector space~$V$ together with a homomorphism $h \colon \bbS \to \GL(V)_\bbR$ such that $h\circ \wt \colon \bbG_{\mult,\bbR} \to \GL(V)_\bbR$ is given by $z \mapsto z^{-n} \cdot \id$. We follow the convention that an element $(z,w) \in \bbC^* \times \bbC^* = \bbS(\bbC)$ acts on the summand $V^{p,q} \subset V_\bbC$ in the Hodge decomposition of~$V_\bbC$ as multiplication by $z^{-p}w^{-q}$. Instead of giving the homomorphism~$h$, we can also describe the Hodge strucuture on~$V$ by giving the corresponding cocharacter $\mu = (h_\bbC \circ i) \colon \bbG_{\mult,\bbC} \to \GL(V)_\bbC$ that is given by the rule that $z \in \bbC^*$ acts on~$V^{p,q}$ as multiplication by~$z^{-p}$.

\subsection{} 
By a Shimura datum we mean a pair $(G,\cY)$ where $G$ is a connected reductive group over~$\bbQ$ and $\cY$ is a $G(\bbR)$-conjugacy class of homomorphisms $\bbS \to G_\bbR$, such that the conditions (2.1.1.1--3) of \cite{DelCorvalis}, Section~2.1 are satisfied. The weight of a Shimura datum is the homomorphism $h\circ \wt \colon \bbG_{\mult,\bbR} \to G_\bbR$, which in fact takes values in the connected centre of~$G_\bbR$ and is independent of $h \in \cY$. In all cases of interest for us, this weight cocharacter is defined over~$\bbQ$. 

If $(G_1,\cY_1)$ and $(G_2,\cY_2)$ are Shimura data then by an embedding $j \colon (G_1,\cY_1) \hookrightarrow (G_2,\cY_2)$ we mean an injective homomorphism $j \colon G_1 \hookrightarrow G_2$ such that composition with~$j$ gives a map $\cY_1 \to \cY_2$.

\subsection{}
Let $X$ be a complex abelian variety. Write $V = H_1(X,\bbQ)$, and let $h \colon \bbS \to \GL(V)_\bbR$ be the homomorphism that gives the Hodge structure on~$V$. By definition, the Mumford--Tate group of~$X$ is the smallest algebraic subgroup $H \subset \GL(V)$ such that $h$ factors through~$H_\bbR$. 

Let $G$ be the Mumford--Tate group of~$X$. If $\cY$ is the $G(\bbR)$-conjugacy class of the homomorphism $h\colon \bbS \to G_\bbR$, the pair $(G,\cY)$ is a Shimura datum whose weight is defined over~$\bbQ$. We refer to it as the Shimura datum given by~$X$. The goal of this section is to study complex abelian varieties~$X$ with associated Shimura datum $(G,\cY)$ such that $\dim(\cY) = 1$.

\subsection{}\label{subsec:minuscule}
Let $Z = Z(G)^0$ be the identity component of the centre of~$G$, which is a torus over~$\bbQ$. The natural homomorphism $Z \times G^\sconn \to G$ is an isogeny. We view~$V$ as a representation of $Z \times G^\sconn$. The natural map $\End^0(X) \to \End(V)$ induces an isomorphism $\End^0(X) \isomarrow \End_{Z \times G^\sconn}(V)$.

Let $H_1,\ldots,H_s$ be the simple factors of~$G_\bbC$, so that $G_\bbC = Z_\bbC \cdot H_1 \cdots H_r$. Every irreducible $G_\bbC$-submodule of~$V_\bbC$ is isomorphic to a representation $\chi \boxtimes r_1 \boxtimes \cdots \boxtimes r_s$, where $\chi$ is a character of~$Z_\bbC$ and $r_j$ is an irreducible representation of~$H_j$ ($j=1,\ldots,s$). By a result of Deligne (\cite{DelCorvalis}, Section~1.3) and Serre (\cite{Serre}, \S~3), the highest weight of every nontrivial representation~$r_j$ that occurs is a minuscule weight. In particular, if $H_j$ is of Lie type~$\mathrm{A}_1$, we must have $H_j \cong \SL_2$, and if $r_j$ is nontrivial, it is the $2$-dimensional standard representation.

\subsection{}\label{subsec:GadDad}
Let $X$ be a complex abelian variety such that the associated Shimura datum $(G,\cY)$ has the property that $\dim(\cY) = 1$, which is equivalent to the assumption that
\begin{equation}\label{eq:Gadrk1}
G^\ad_\bbR \cong \PGL_{2,\bbR} \times \text{compact factors}\, .
\end{equation}
If this holds, there exists a totally real field~$F$ and a $4$-dimensional central simple $F$-algebra~$D$ (unique up to isomorphism) such that
\[
D \otimes_\bbQ \bbR \cong M_2(\bbR) \times \text{a product of factors $\bbH$}\, ,
\]
and such that $G^\sconn \cong \Res_{F/\bbQ}\; \SL_{1,D}$. Let $\sigma_{\nc} \in \Emb(F)$ be the unique embedding with the property that $D \otimes_{F,\sigma_{\nc}} \bbR \cong M_2(\bbR)$. Then 
\begin{equation}\label{eq:GscExplicit}
G^\sconn_\bbR \cong \SL_{2,\bbR} \times \prod_{\substack{\sigma \in \Emb(F)\\ \sigma \neq \sigma_{\nc}}}\; \SL_{1,\bbH}\, ,\qquad
G^\ad_\bbR \cong \PGL_{2,\bbR} \times \prod_{\substack{\sigma \in \Emb(F)\\ \sigma \neq \sigma_{\nc}}}\;\PGL_{1,\bbH}\, .
\end{equation}
(Here $\SL_{1,\bbH} = \SU(2)$ is the compact real form of~$\SL_2$.) In the adjoint Shimura datum $(G^\ad,\cY^\ad)$, the space~$\cY^\ad$ is the $G^\ad(\bbR)$-conjugacy class of the homomorphism $\bbS \to G^\ad_\bbR$ that on the first factor is given on $\bbR$-valued points by
\begin{equation}\label{eq:h(a+bi)}
a+bi \mapsto \left[\begin{smallmatrix} a&-b \\ b&a\end{smallmatrix}\right]
\end{equation}
and that is trivial on the compact factors. 

Because the centre of~$G^\sconn$ is $2$-torsion, so is the kernel of the isogeny $Z \times G^\sconn \to G$. Hence, there exists a unique homomorphism $\tilde{h} = (\tilde{h}_{\mathrm{c}},\tilde{h}_{\mathrm{s}}) \colon \bbS \to Z_\bbR \times G^\sconn_\bbR$ such that the diagram
\begin{equation}\label{eq:htildediag}
\begin{tikzcd}
\bbS \ar[d,"z \mapsto z^2"'] \ar[r,"\tilde{h}"] & Z_\bbR \times G^\sconn_\bbR \ar[d]\\
\bbS \ar[r,"h"] & G_\bbR
\end{tikzcd}
\end{equation}
is commutative. Define $\tilde{\mu} = \tilde{h}_\bbC \circ i$, and write it as $\tilde{\mu} = (\tilde{\mu}_{\mathrm{c}},\tilde{\mu}_{\mathrm{s}}) \colon \bbG_{\mult,\bbC} \to Z_\bbC \times G^\sconn_\bbC$. ($\mathrm{c} = \text{centre}$, $\mathrm{s} = \text{semisimple}$.)

\subsection{}\label{subsec:Xsimple}
In addition to the assumptions made in~\ref{subsec:GadDad}, assume that $X$ is a simple abelian variety of dimension~$g$. The endomorphism algebra $\End^0(X)$ is a division algebra of finite dimension over~$\bbQ$. Let $E$ be the centre of~$\End^0(X)$, which is either a totally real field (Albert Types I, II and~III) or a CM field (Albert Type~IV). Then $V_\bbC$ is a module over $E \otimes_\bbQ \bbC = \prod_{\phi \in \Emb(E)}\; \bbC$; correspondingly, we have a decomposition
\begin{equation}\label{eq:VCdec}
V_\bbC = \bigoplus_{\phi \in \Emb(E)}\; V_\phi\, ,
\end{equation}
with $V_\phi = V\otimes_{E,\phi} \bbC$. With the notation $V_\phi^{p,q} = V^{p,q} \cap V_\phi$ (intersection taken inside~$V_\bbC$) we then have
\[
V_\phi = V^{-1,0}_\phi \oplus V^{0,-1}_\phi\, ,\qquad
V^{p,q} = \bigoplus_{\phi \in \Emb(E)}\; V_\phi^{p,q}\quad \text{(for $(p,q) = (-1,0)$ or $(0,-1)$)}\, ,
\]
and complex conjugation on~$V_\bbC$ interchanges the summands $V^{-1,0}_\phi$ and~$V^{0,-1}_{\bar\phi}$. The function
\[
\ff \colon \Emb(E) \to \bbN\quad\text{given by}\quad \ff(\phi) = \dim_\bbC(V^{-1,0}_\phi)
\]
satisfies $\ff(\phi) + \ff(\bar\phi) = \dim_E(V) = \frac{2g}{[E:\bbQ]}$ for all $\phi \in \Emb(E)$. In particular, if $E$ is totally real then $[E:\bbQ]$ divides~$g$ and $\ff$ is constant with value~$\frac{g}{[E:\bbQ]}$. We refer to~$\ff$ as the multiplication type of~$X$.

The action of~$E$ on~$V$ gives a representation $T_E \to \GL(V)$, which is isomorphic to a sum of $\dim_E(V)$ copies of the standard representation~$\St_E$ (see Example~\ref{exa:TK}). Viewing $T_E$ as a subgroup of~$\GL(V)$ via this homomorphism, the connected centre~$Z$ of the Mumford--Tate group~$G$ is a subgroup of~$T_E$. With notation as in Example~\ref{exa:TK}, the character group~$\mathsf{X}^*(Z)$ of~$Z$ is therefore a quotient of~$\mathsf{X}^*(T_E) = \oplus_{\phi \in \Emb(E)}\; \bbZ\cdot \mathrm{e}_\phi$, and the cocharacter group~$\mathsf{X}_*(Z)$ is a primitive subgroup of $\mathsf{X}_*(T_E) = \oplus_{\phi \in \Emb(E)}\; \bbZ\cdot \check{\mathrm{e}}_\phi$.

\begin{proposition}\label{prop:repdescr}
Let the notation and assumptions be as above.
\begin{enumerate}
\item\label{repdescr1} The cocharacter $\tilde{\mu}_{\mathrm{c}} \colon \bbG_{\mult,\bbC} \to Z_\bbC \subset T_{E,\bbC}$ corresponds to the element 
\[
\sum_{\phi \in \Emb(E)}\;  \frac{2\cdot \ff(\phi)}{n} \cdot \check{\mathrm{e}}_\phi = \sum_{\phi \in \Emb(E)}\; \frac{[E:\bbQ] \cdot \ff(\phi)}{g}\cdot \check{\mathrm{e}}_\phi
\]
of~$\mathsf{X}_*(T_E)$, where $n = \dim_E(V) = \frac{2g}{[E:\bbQ]}$.

\item\label{repdescr2} With identifications as in~\eqref{eq:GscExplicit}, the homomorphism $\tilde{\mu}_{\mathrm{s}}\colon \bbG_{\mult,\bbC} \to G^\sconn_\bbC$ is conjugate under~$G^\ad(\bbR)$ to the homomorphism given on $\bbC$-valued points by
\begin{equation}\label{eq:tildemus}
z \mapsto \left(\begin{pmatrix} \frac{z^2+1}{2z} & i \cdot \frac{z^2-1}{2z} \\ -i \cdot \frac{z^2-1}{2z} & \frac{z^2+1}{2z} \end{pmatrix}, 1,\ldots,1\right)\, .
\end{equation}

\item\label{repdescr3} The representation $Z \to \GL(V)$ is a direct sum of copies of (the restriction to $Z \subset T_E$ of) the standard representation~$\St_E$. (Notation as in Example~\ref{exa:TK}.)

\item\label{repdescr4} There exists a $\Gamma_\bbQ$-orbit~$\biI$ of nonempty subsets of~$\Emb(F)$ such that the representation $G^\sconn \to \GL(V)$ is isotypical of type~$\rho_{\biI}$. (Notation as in Example~\ref{exa:rhoIrep} and Remark~\ref{rem:rhoI}, applied with $\cG = G^\sconn$.)
\end{enumerate}
\end{proposition}

In \ref{repdescr2}, note that $G^\ad$ acts by conjugacy on the space of cocharacters of~$G^\sconn$.

\begin{proof}
\ref{repdescr1} We have $G \subset \GL_E(V)$. Let $\det_E \colon \GL_E(V) \to T_E$ be the $E$-linear determinant. The composition
\begin{equation}\label{eq:TEdet}
T_E \times G^\sconn \to G_E \xrightarrow{~\det_E~} T_E
\end{equation}
is given by $(t,y) \mapsto t^n$. Next consider the composition
\[
f \colon \bbG_{\mult,\bbC} \xrightarrow{~\tilde{\mu}~} T_{E,\bbC} \times G^\sconn_\bbC \to \GL_E(V)_\bbC = \prod_{\phi \in \Emb(E)}\; \GL(V_\phi)\, .
\]
By definition of the function~$\ff$, the action of $z \in \bbC^*$ on $V_\phi$ through this homomorphism is conjugate to the homomorphism 
\[
z \mapsto \diag\bigl(\underbrace{z^2,\ldots,z^2}_{\text{$\ff(\phi)$ terms}},1,\ldots,1\bigr)\, .
\]
(The $z^2$ comes from the fact that $\tilde{\mu}$ is a lift of the square of~$\mu$, cf.\ diagram~\eqref{eq:htildediag}.) In particular, $\det_E \circ f \colon \bbG_{\mult,\bbC} \to T_{E,\bbC} = \prod_{\phi \in \Emb(E)}\; \bbG_{\mult,\bbC}$ is given on the factor indexed by~$\phi$ by $z \mapsto z^{2\cdot \ff(\phi)}$. By our description of the map~\eqref{eq:TEdet}, if $\tilde{\mu}_{\mathrm{c}}$ corresponds to the element $\sum_{\phi \in \Emb(E)}\; a_\phi\cdot \check{\mathrm{e}}_\phi \in \mathsf{X}_*(T_E)$, we find the relation $n\cdot a_\phi = 2\cdot \ff(\phi)$.

\ref{repdescr2} On $\bbC$-valued points, the homomorphism $\bbS \to \PGL_{2,\bbR}$ of~\eqref{eq:h(a+bi)} is given by
\[
(z,w) \mapsto \left[\begin{smallmatrix} z+w & i(z-w) \\ -i(z-w) & z+w \end{smallmatrix}\right]
\]
and therefore the homomorphism $\ad \circ \tilde{\mu}_{\mathrm{s}} \colon \bbG_{\mult,\bbC} \to G^\ad_\bbC$ is conjugate to the cocharacter that, under the identification~\eqref{eq:GscExplicit}, is given by 
\[
z \mapsto \left(\Bigl[\begin{smallmatrix} z^2+1 & i(z^2-1) \\ -i(z^2-1) & z^2+1 \end{smallmatrix}\Bigr],1,\ldots,1\right)\; .
\]
Because a cocharacter of~$G^\ad_\bbC$ admits at most one lift to~$G^\sconn_\bbC$, it now suffices to remark that~\eqref{eq:tildemus} indeed lifts the latter homomorphism.

\ref{repdescr3} This is obvious.

\ref{repdescr4} This follows from what was explained in Section~\ref{subsec:minuscule}.
\end{proof}

\begin{proposition}\label{prop:Albert123}
Let $X$ be a $g$-dimensional simple complex abelian variety such that in the associated Shimura datum $(G,\cY)$ we have $\dim(\cY) = 1$. Assume that $X$ is of Albert type \textup{I}, \textup{II} or~\textup{III}. Let notation be as in Sections~\ref{subsec:GadDad}--\ref{subsec:Xsimple}.
\begin{enumerate}
\item\label{Albert123-1} We have $Z = \bbG_{\mult} \cdot \id_V$.

\item\label{Albert123-2} With $G^\sconn = \Res_{F/\bbQ}\; \SL_{1,D}$ as in~\ref{subsec:GadDad}, the representation $G^\sconn \to \GL(V)$ is irreducible and is the corestriction representation as in Example~\ref{exa:111Rep}, with $\cG = G^\sconn$. (In other words, it is the irreducible representation~$\rho_I$ of Example~\ref{exa:rhoIrep} with $I = \Emb(F)$.)

\item\label{Albert123-3} Let $m=[F:\bbQ]$. We are in one of the following three cases:
\begin{enumerate}[label=\textup{(\Roman*)}]
\item $m$ is odd, $\Cores_{F/\bbQ}(D) \cong M_{2^m}(\bbQ)$ and $\End^0(X) \cong \bbQ$;

\item $m$ is odd, $\Cores_{F/\bbQ}(D) \not\cong M_{2^m}(\bbQ)$ and $\End^0(X)$ is a quaternion algebra over~$\bbQ$ that splits over~$\bbR$;

\item $m$ is even, $\Cores_{F/\bbQ}(D) \not\cong M_{2^m}(\bbQ)$ and $\End^0(X)$ is a quaternion algebra over~$\bbQ$ that does not split over~$\bbR$.
\end{enumerate}
We have $g=2^{m-1}$ in case~\textup{(I)} and $g=2^m$ in cases \textup{(II)} and~\textup{(III)}.
\end{enumerate}
\end{proposition}

(In~\ref{Albert123-3}, the labels correspond to the Albert type of~$X$.)

\begin{proof}
\ref{Albert123-1} This is well-known; see for instance \cite{Tankeev}, Lemma~1.4. We can also see it directly: because $E$ is totally real, the function~$\ff$ is constant, and it follows from Proposition~\ref{prop:repdescr}\ref{repdescr1} that the cocharacter~$\tilde{\mu}_{\mathrm{c}}$ is defined over~$\bbQ$. Now use that $Z$ is the smallest subtorus of~$T_E$ such that $\tilde{\mu}_{\mathrm{c}}$ factors through~$Z_\bbC$.

\ref{Albert123-2} Because $X$ is simple, $V$ is irreducible as a representation of $Z \times G^\sconn$, and then it follows from~\ref{Albert123-1} that it is irreducible as a representation of~$G^\sconn$. By Proposition~\ref{prop:repdescr}\ref{repdescr4}, $V$ is therefore a representation of the form~$\rho_{\biI}$ for some $\Gamma_\bbQ$-orbit~$\biI$ of nonempty subsets of~$\Emb(E)$. We need to show that only $\biI = \bigl\{\Emb(E)\bigr\}$ is possible. To see this, first note (see Example~\ref{exa:rhoIrep}, and use that $F$ is totally real) that $V_\bbR$, as a representation of~$G^\sconn_\bbR$, decomposes as a sum of representations~$R_{\bbR,I}$, for $I \in \biI$, such that $(R_{\bbR,I})_\bbC$ is a sum of copies of $\boxtimes_{\sigma \in I}\; \St_\sigma$. If $\biI \neq \bigl\{\Emb(F)\bigr\}$ then there exists some $I \in \biI$ such that $\sigma_{\nc} \notin I$, where we recall that $\sigma_{\nc} \in \Emb(F)$ is the unique embedding such that $D \otimes_{F,\sigma_{\nc}} \bbR \cong M_2(\bbR)$. Because the homomorphism $h\colon \bbS \to G_\bbR$ that defines the Hodge structure on~$V$ projects trivally onto the compact factors of~$G^\ad_\bbR$, it follows that the (real) Hodge structure on the direct summand of~$V_\bbR$ that corresponds to the representation~$R_{\bbR,I}$ is of Tate type. This is impossible, as $V_\bbR$ is a real Hodge structure of type $(-1,0) + (0,-1)$.

\ref{Albert123-3} By~\ref{Albert123-2}, we are in the situation of Example~\ref{exa:111Rep} with $k=\bbQ$ and $L = F$. If $\cE$ is the division algebra with centre $k=\bbQ$ that represents the class of $\Cores_{F/\bbQ}\; D$ (as in that example), we have $\cE = \End(\rho) = \End^0(X)$. But we have seen that either $\cE = \bbQ$ or $\cE$ is a quaternion algebra over~$\bbQ$. On the other hand, by looking at the invariants of~$D$ at the infinite places of~$F$, we see that $\inv_\infty(\cE) = 0 \in \Br(\bbR)$ if $m$ is odd and $\inv_\infty(\cE) = \frac{1}{2} \in \Br(\bbR)$ if $m$ is even. It readily follows that the listed cases (I)--(III) are the only three possibilities. Finally, the given recipe for~$g$ is just the calculation of the dimension of the corestriction representation.
\end{proof}

\subsection{}\label{subsec:Albert4}
Returning to the setting of~\ref{subsec:GadDad}, we now assume that $X/\bbC$ is a simple abelian variety of Albert Type~IV, which means that~$E$, the centre of~$\End^0(X)$, is a CM field. Let $E_0 \subset E$ be the maximal totally real subfield. As before, we fix~$F$ and~$D$ and an identification of~$G^\sconn$ with $\Res_{F/\bbQ}\; \SL_{1,D}$. 

It will be convenient to view $V$ as a representation of $T_E \times G^\sconn$. Because $E$ is the centre of $\End^0(X)$ and $Z \subset T_E$, we have $\End^0(X) = \End_{Z\times G^\sconn}(V) = \End_{T_E \times G^\sconn}(V)$.

Let notation be as in Example~\ref{exa:rhoIrep} and Remark~\ref{rem:rhoI}, with $\cG = G^\sconn$. By Proposition~\ref{prop:repdescr}\ref{repdescr4}, there exists a $\Gamma_\bbQ$-orbit~$\biI$ of nonempty subsets of~$\Emb(F)$ such that $V$ is isotypical of type~$\rho_{\biI}$ as a $G^\sconn$-module. Realise this representation as $\rho_{\biI} \colon G^\sconn \to \GL(W_{\biI})$ for some $\bbQ$-vector space~$W_{\biI}$. Recall (see Remark~\ref{rem:rhoI}) that we write $\cE_{\biI} = \End_{G^\sconn}(W_{\biI})$, and that the centre of~$\cE_{\biI}$ is called~$k_{\biI}$. Either $\cE_{\biI} = k_{\biI}$ or $\cE_{\biI}$ is a quaternion algebra over~$k_{\biI}$.
 
With this notation, $W_{\biI}$ has the structure of a left $\cE_{\biI}$-module. This induces the structure of a right $\cE_{\biI}$-module on the space
\[
H = \Hom_{G^\sconn}(W_{\biI},V)\, .
\]
The torus $T_E$ acts on~$H$ by $\cE_{\biI}$-linear automorphisms, through its action on~$V$. The evaluation map gives a $T_E \times G^\sconn$-equivariant isomorphism $H \otimes_{\cE_{\biI}} W_{\biI} \isomarrow V$, where $G^\sconn$ acts on $H \otimes_{\cE_{\biI}} W_{\biI}$ via $ \id_H \otimes \rho_{\biI} $ and $T_E$ acts via its action on~$H$. This gives us an isomorphism
\begin{equation}\label{eq:EndGscV}
\End_{\cE_{\biI}}(H) \isomarrow \End_{G^\sconn}(V)\, .
\end{equation}
Note that $\End_{\cE_{\biI}}(H)$ is a central simple $k_{\biI}$-algebra.

\begin{lemma}\label{lem:doublecentral} 
Notation and assumptions as above.
\begin{enumerate}
\item\label{dc1} Identify $\End^0(X)$ with the $\bbQ$-subalgebra $\End_{T_E \times G^\sconn}(V)$ of~$\End_{G^\sconn}(V)$, and view the latter as a $k_{\biI}$-algebra via the isomorphism~\eqref{eq:EndGscV}. Then $k_{\biI} \subset E_0$, and hence $E_0$, $E$ and~$\End^0(X)$ are $k_{\biI}$-subalgebras of~$\End_{G^\sconn}(V)$. Further, $\End^0(X)$ is the centralizer of~$E$ in $\End_{G^\sconn}(V)$, and $E$ is the centralizer of~$\End^0(X)$.

\item\label{dc2} The multiplication type $\ff \colon \Emb(E) \to \bbN$ is not constant. 
\end{enumerate}
\end{lemma}

\begin{proof} 
(1) The action of~$T_E$ on~$V$ commutes with the action of~$G^\sconn$, and $\End^0(X) \subset \End_{G^\sconn}(V)$ is the subalgebra of elements that commute with the action of~$T_E$. Therefore, the centre of $\End_{G^\sconn}(V)$, which is~$k_{\biI}$, is contained in the centre of~$\End^0(X)$, which is~$E$. Since $k_{\biI}$ is totally real, even $k_{\biI} \subset E_0$. Moreover, the centralizer of~$E$ is contained in~$\End^0(X)$; but $E$ is the centre of~$\End^0(X)$, so in fact the centralizer of~$E$ equals $\End^0(X)$. The last assertion then follows by the double centralizer theorem.

(2) Suppose $\ff$ were constant. As in the proof of Proposition~\ref{prop:Albert123}\ref{Albert123-1}, this would give $Z = \bbG_{\mult}$, and hence $\End^0(X) \cong \End_{G^\sconn}(V)$. But the centre of $\End_{G^\sconn}(V)$ is~$k_{\biI}$, which is totally real and therefore cannot be equal to~$E$; contradiction. 
\end{proof}

\subsection{}\label{subsec:threecases}
Let notation and assumptions be as in~\ref{subsec:Albert4}. In addition to its right $\cE_{\biI}$-module structure, $H$ has the structure of an $E$-vector space through the action of~$E$ on~$V$. Clearly, the $E$-action on~$H$ commutes with the $\cE_{\biI}$-action. The embedding $\iota \colon k_{\biI} \hookrightarrow E_0 \subset E$ of Lemma~\ref{lem:doublecentral}(1) is such that $f \circ a = \iota(a) \circ f$ for every $a \in k_{\biI}$ and $f \in H$, so the $\cE_{\biI}$-action and the $E$-action induce the same structure of a $k_{\biI}$-vector space on~$H$. Hence, $H$ has the structure of a right $E \otimes_{k_{\biI}} \cE_{\biI}$-module. Note that $E \otimes_{k_{\biI}} \cE_{\biI}$ is a central simple $E$-algebra; if we write $E\cE_{\biI} \subset \End_{k_{\biI}}(H)$ for the $k_{\biI}$-subalgebra generated by $E$ and~$\cE_{\biI}$, the natural map $E \otimes_{k_{\biI}} \cE_{\biI} \twoheadrightarrow E\cE_{\biI}$ is therefore an isomorphism. Via the isomorphism~\eqref{eq:EndGscV}, we find that $\End^0(X) \cong \End_{E\cE_{\biI}}(H)$, and because $\End^0(X)$ is a division algebra, only three cases are possible:
\newlength{\caselabelwd}\settowidth{\caselabelwd}{Case~2.\quad}
\begin{description}[font=\normalfont,labelindent=0pt,labelwidth=\the\caselabelwd,labelsep=0pt,leftmargin=\the\caselabelwd]
\item[Case~0.] $\cE_{\biI} = k_{\biI}$ and $\dim_E(H) = 1$;

\item[Case~1.] $\cE_{\biI}$ is a quaternion algebra over~$k_{\biI}$ and $E\cE_{\biI} \cong M_2(E)$, in which case $\dim_E(H) = 2$;

\item[Case~2.] $\cE_{\biI}$ is a quaternion algebra over~$k_{\biI}$ and $E\cE_{\biI}$ is a quaternion algebra over~$E$, in which case $H$ is free of rank~$1$ over~$E\cE_{\biI}$, and hence $\dim_E(H) = 4$.
\end{description}
(As we shall see below, Case~0 in fact does not occur.)

Write $n = \dim_E(V) = \frac{2g}{[E:\bbQ]}$. Recall (see Remark~\ref{rem:rhoI}) that $\Emb(k_{\biI})$ is in bijection with~$\biI$. The inclusions $k_{\biI} \subset E_0 \subset E$ give rise to natural maps $\Emb(E) \to \Emb(E_0) \to \Emb(k_{\biI})$. For $\phi$ in $\Emb(E)$ or~$\Emb(E_0)$, we write $\phi|_{k_{\biI}}$ for its image in~$\Emb(k_{\biI})$, and we write $I_\phi \in \biI$ for the corresponding subset. Recall that $\ell(\biI)$ denotes the cardinality of the sets $I \in \biI$. We find the following:
\begin{description}[font=\normalfont,labelindent=0pt,labelwidth=\the\caselabelwd,labelsep=0pt,leftmargin=\the\caselabelwd]
\item[Case~0.] $V \cong \St_E \boxtimes_{k_{\biI}} W_{\biI}$ as representations of $T_E \times G^\sconn$;

$\End^0(X) = E$, with $n = 2^{\ell(\biI)}$ and $g = [E:\bbQ] \cdot 2^{\ell(\biI)-1}$.
\smallskip

\item[Case 1.] $V^{\oplus 2} \cong \St_E \boxtimes_{k_{\biI}} W_{\biI}$ as representations of $T_E \times G^\sconn$;

$\End^0(X) = E$, with $n = 2^{\ell(\biI)}$ and $g = [E:\bbQ] \cdot 2^{\ell(\biI)-1}$.
\smallskip

\item[Case~2.] $V \cong \St_E \boxtimes_{k_{\biI}} W_{\biI}$ as representations of $T_E \times G^\sconn$;

$\End^0(X) = E\cE_{\biI}$, with $n = 2^{\ell(\biI)+1}$ and $g = [E:\bbQ] \cdot 2^{\ell(\biI)}$.
\end{description}
We have
\begin{equation}\label{eq:StEWItensorC}
(\St_E \boxtimes_{k_{\biI}} W_{\biI}) \otimes_\bbQ \bbC 
= \bigoplus_{I \in \biI}\; \Bigl(\mathop{\oplus}\limits_{\substack{\phi\in \Emb(E)\\ I_\phi = I}}\; \bbC \Bigr) \mathop{\otimes}\limits_\bbC \Bigl(\mathop{\boxtimes}\limits_{\sigma \in I}\; \St_\sigma\Bigr)^{\oplus \deg(\cE_{\biI})}
= \bigoplus_{\phi \in \Emb(E)}\; \Bigl(\mathop{\boxtimes}\limits_{\sigma \in I_\phi}\; \St_\sigma\Bigr)^{\oplus \deg(\cE_{\biI})}\, .
\end{equation}
We consider the action of~$\bbG_{\mult,\bbC}$ on this space via the homomorphism $\tilde\mu \colon \bbG_{\mult,\bbC} \to T_{E,\bbC} \times G^\sconn_\bbC$, which is described by Proposition~\ref{prop:repdescr}. Note that the cocharacter~\eqref{eq:tildemus} is $G^\sconn(\bbC)$-conjugate to the cocharacter given by
\[
z \mapsto \left(\Bigl(\begin{smallmatrix} z & 0 \\ 0 & z^{-1}\end{smallmatrix}\Bigr), 1,\ldots,1\right)\, .
\]
It follows that in the decomposition~\eqref{eq:StEWItensorC}, the $\bbG_{\mult,\bbC}$-action on the summand indexed by $\phi \in \Emb(E)$ has weights $\frac{2\cdot \ff(\phi)}{n} + 1$ and $\frac{2\cdot \ff(\phi)}{n} - 1$ if $\sigma_{\nc} \in I_\phi$, and has weight $\frac{2\cdot \ff(\phi)}{n}$ if $\sigma_{\nc} \notin I_\phi$. (Recall that $\sigma_{\nc} \in \Emb(F)$ is the unique embedding for which $D \otimes_{F,\sigma_{\nc}} \bbR \cong M_2(\bbR)$.) Note that whether or not $\sigma_{\nc}$ is in~$I_\phi$ only depends on~$\phi|_{k_{\biI}}$, and is therefore invariant under complex conjugation.

Since we know that $\bbG_{\mult,\bbC}$ acts on~$V$ with weights $0$ and~$2$, we find that
\begin{equation}
\bigl\{\ff(\phi),\ff(\bar\phi)\bigr\} = \{0,n\}\quad \text{if $\sigma_{\nc} \notin I_\phi$,}\qquad
\ff(\phi) = \ff(\bar\phi) = \tfrac{n}{2}\quad \text{if $\sigma_{\nc} \in I_\phi$.}
\end{equation}
It follows from Lemma~\ref{lem:doublecentral}\ref{dc2} that $\biI \neq \{\Emb(F)\}$. By Lemma~\ref{lem:Eisquat}, Case~0 is therefore excluded.
\medskip

The following proposition summarizes what we have found.

\begin{proposition}\label{prop:Albert4}
Let $X$ be a $g$-dimensional simple complex abelian variety with associated Shimura datum $(G,\cY)$ such that $\dim(\cY) = 1$. Assume that $X$ is of Albert type~\textup{IV}. Let notation be as in Sections~\ref{subsec:GadDad}--\ref{subsec:Xsimple}, and recall that we write $n = 2g/[E:\bbQ]$. There exists a $\Gamma_\bbQ$-orbit~$\biI$ of nonempty proper subsets $I \subsetneq \Emb(F)$ such that, with notation as in~\ref{subsec:Albert4}--\ref{subsec:threecases}, $E_0$ is a field extension of~$k_{\biI}$, and such that we are in one of the following cases:
\begin{enumerate}[label=\textup{\arabic*.}]
\item $\cE_{\biI}$ is a quaternion algebra over~$k_{\biI}$ such that $E \otimes_{k_{\biI}} \cE_{\biI} \cong M_2(E)$, in which case $\End^0(X) = E$ and $g = [E:\bbQ] \cdot 2^{\ell(\biI)-1}$;

\item $\cE_{\biI}$ is a quaternion algebra over~$k_{\biI}$ such that $E \otimes_{k_{\biI}} \cE_{\biI}$ is a quaternion algebra over~$E$, in which case $\End^0(X) \cong E \otimes_{k_{\biI}} \cE_{\biI}$ and $g = [E:\bbQ] \cdot 2^{\ell(\biI)}$.
\end{enumerate}
Moreover, if $\Phi_0 = \bigl\{\phi_0 \in \Emb(E_0) \bigm| \sigma_{\nc} \notin I_{\phi_0} \bigr\}$, there exists a subset $\Phi \subset \Emb(E)$ with the property that the restriction map $\End(E) \twoheadrightarrow \Emb(E_0)$ induces a bijection $\Phi \isomarrow \Phi_0$, and such that the multiplication type $\ff \colon \Emb(E) \to \bbN$ is given by
\begin{equation}\label{eq:multtype}
\ff(\phi) = \begin{cases}
n & \text{if $\phi \in \Phi$;}\\
0 & \text{if $\bar\phi \in \Phi$;}\\
\frac{n}{2} & \text{otherwise.}\\
\end{cases}
\end{equation}
\end{proposition}

\subsection{Remark.}\label{rem:explanation} The following explanation may help to understand what is going on. We have the decomposition $V_\bbC = \oplus_{\phi \in \Emb(E)}\, V_\phi$ as in~\eqref{eq:VCdec}, with $\dim_\bbC(V_\phi) = n$, which is even. The Hodge decomposition of~$V_\bbC$ is described by the action of~$\bbG_{\mult,\bbC}$ via the homomorphism $\tilde{\mu} = (\tilde{\mu}_{\mathrm{c}},\tilde{\mu}_{\mathrm{s}})$. Here $\tilde{\mu}_{\mathrm{c}}$ is a homomorphism $\bbG_{\mult,\bbC} \to T_{E,\bbC} = \prod_{\phi \in \Emb(E)}\; \bbG_{\mult,\bbC}$, and we see that the $\tilde{\mu}_{\mathrm{c}}$-action of $z \in \bbC^*$ on the summand $V_\phi \subset V_\bbC$ is multiplication by $z^{m(\phi)}$ for some integer~$m(\phi)$. If $\phi$ is such that $\sigma_{\nc} \notin I_\phi$, the action of~$\bbG_{\mult,\bbC}$ on~$V_\phi$ via the homomorphism~$\tilde{\mu}_{\mathrm{s}}$ is trivial; hence, $V_\phi$ is entirely of Hodge type $(-1,0)$ or entirely of type $(0,-1)$. Since~$\tilde\mu$ is the square of the usual cocharacter~$\mu$, this means that $m(\phi) = 0$ or $m(\phi) = 2$, which by Proposition~\ref{prop:repdescr}\ref{repdescr1} is equivalent to: $\ff(\phi) = 0$ or $\ff(\phi) = n$. As we shall discuss below, the set~$\Phi$ in Proposition~\ref{prop:Albert4} may be thought of as a `partial CM type'; it keeps track of whether, for embeddings~$\phi$ with $\sigma_{\nc} \notin I_\phi$, the type of~$V_\phi$ is $(-1,0)$ or $(0,-1)$.

If $\sigma_{\nc} \in I_\phi$ the situation is very different: in this case, the action of~$\bbG_{\mult,\bbC}$ on~$V_\phi$ via the homomorphism~$\tilde{\mu}_{\mathrm{s}}$ is nontrivial and has weights~$1$ and~$-1$, both with multiplicity $\frac{n}{2}$. (Informally speaking, the Hodge decomposition of~$V_\phi$ comes from the semisimple part of the Mumford--Tate group.) In this case, the $\tilde{\mu}_{\mathrm{c}}$-action of $\bbG_{\mult,\bbC}$ on~$V_\phi$ only shifts these weights, which means that $m(\phi) = m(\bar\phi) = 1$ (equivalently: $\ff(\phi) = \ff(\bar\phi) = \frac{n}{2}$), and there is no further bookkeeping to be done.

\subsection{Example.}\label{exa:unitary} Let $g$ be an even positive integer. Let $(X,\lambda)$ be a $g$-dimensional polarised complex abelian variety such that $\End^0(X) = E$ is a CM field of degree~$g$ over~$\bbQ$, with maximal totally real subfield~$E_0$. Let $V = H_1(X,\bbQ)$, which is a $2g$-dimensional $\bbQ$-vector space on which~$E$ acts, and which can therefore also be viewed as a $2$-dimensional $E$-vector space. 

There is a unique $(-1)$-hermitian form $\Psi \colon V \times V \to E$ such that the Riemann form of~$\lambda$ equals $\Tr_{E/\bbQ} \circ \Psi$. Let $*$ be the involution of $\End_E(V)$ such that $\Psi\bigl(f(x),y\bigr) = \Psi\bigl(x,f^*(y)\bigr)$ for all $f \in \End_E(V)$ and $x$, $y \in V$. Then 
\[
D = \bigl\{f \in \End_E(V)\bigm| f^* = f\bigr\}
\]
is a quaternion algebra over~$E_0$. The natural homomorphism $E \otimes_{E_0} D \to \End_E(V)$ is an isomorphism, and under this isomorphism $*$ corresponds to $\iota \otimes \dagger$, where $\iota$ is complex conjugation on~$E$ and $\dagger$ is the canonical involution of~$D$. Define $\UU(V,\Psi) = \Res_{E_0/\bbQ}\; \UU(V/E,\Psi)$, where by $\UU(V/E,\Psi) \subset \GL_E(V)$ we mean the unitary group of~$\Psi$, which is a form of~$\GL_2$ over~$E_0$. Further define $\GU(V,\Psi) = \bbG_{\mult} \cdot \UU(V,\Psi)$, where $\bbG_{\mult} = \bbG_{\mult,\bbQ}\cdot \id_V \subset \GL(V)$. The isomorphism $E \otimes_{E_0} D \isomarrow \End_E(V)$ gives rise to an isomorphism $\Res_{E_0/\bbQ}\; \GL_{1,D} \isomarrow \UU(V,\Psi)$.

For $\sigma \in \Emb(E_0)$, write $E_\sigma = E \otimes_{E_0,\sigma} \bbR$, which is non-canonically isomorphic to~$\bbC$. Let $\Psi_\sigma$ be the $E_\sigma$-valued $(-1)$-hermitian form on $V_\sigma = V \otimes_{E_0,\sigma} \bbR$ that is obtained from~$\Psi$ by extension of scalars via~$\sigma$, and let $D_\sigma = D \otimes_{E_0,\sigma} \bbR$. We then have
\[
\UU(V,\Psi)_\bbR \cong \prod_{\sigma \in \Emb(E_0)}\; \UU(V_\sigma,\Psi_\sigma)
\] 
and isomorphisms of real algebraic groups $\GL_{1,D_\sigma} \isomarrow \UU(V_\sigma,\Psi_\sigma)$.

We now assume that there is a unique $\sigma_{\nc} \in \Emb(E_0)$ such that $\Psi_\sigma$ is indefinite for $\sigma = \sigma_{\nc}$ and is definite otherwise. (Indefinite here means that $\UU(V_\sigma,\Psi_\sigma)$ is noncompact.) The Mumford--Tate group~$G$ is a subgroup of $\GU(V,\Psi)$, and the two groups have the same adjoint group. If $(G,\cY)$ is the associated Shimura datum, $\dim(\cY) = 1$. We are therefore in the situation studied above, with $F=E_0$. The representation of $G^\sconn = \Res_{E_0/\bbQ}\; \SL_{1,D}$ on~$V$ is isotypical of type~$\rho_{\biI}$, where $\biI \subset \cP\bigl(\Emb(E_0)\bigr)$ is the set of singletons, which gives $k_{\biI} \cong E_0 = F$ and $\cE_{\biI} = D$. Because $E \otimes_{E_0} D \isomarrow \End_E(V)$ (so: $E$ is a splitting field of~$D$), we are in Case~1 of Proposition~\ref{prop:Albert4}. 

The subset $\Phi \subset \Emb(E)$ is such that the restriction map $\Emb(E) \to \Emb(E_0)$ gives a bijection $\Phi \to \Emb(E_0)\setminus\{\sigma_\nc\}$. In other words, for each $\sigma \neq \sigma_\nc$ in~$\Emb(E_0)$ there is a unique embedding $\phi \in \Emb(E)$ above~$\sigma$ with the property that $V_\phi^{-1,0} \neq 0$, and $\Phi$ is the collection of these~$\phi$. 

{}From the description of the cocharacter~$\mu_{\mathrm{c}}$ given in Proposition~\ref{prop:repdescr}\ref{repdescr1}, we can deduce that $Z \subset T_E$ is the subtorus of elements $y \in E^*$ such that $y\bar{y} \in \bbQ^*$, and it follows that $G = \GU(V,\Psi)$.

\subsection{}\label{subsec:PartCM}
In the situation described in Proposition~\ref{prop:Albert4}, the subset $\Phi \subset \Emb(E)$ determines the multiplication type~$\ff$, which by Proposition~\ref{prop:repdescr}\ref{repdescr1} determines the cocharacter $\tilde\mu_{\mathrm{c}} \colon \bbG_{\mult,\bbC} \to T_{E,\bbC}$. Because $Z \subset T_E$ (the identity component of the centre of the Mumford--Tate group) is the smallest subtorus such that $\tilde\mu_{\mathrm{c}}$ factors through~$Z_\bbC$, we conclude that $\Phi$ determines~$Z$. This centre cannot be too small; for instance, we have seen that $\End_{Z\times G^\sconn}(V) = \End_{T_E \times G^\sconn}(V)$. This gives a nontrivial condition on~$\Phi$. To state it, we introduce the notion of a partial CM type.

\begin{definition}\label{def:PartialCM}
Let $E$ be a CM field with maximal totally real subfield~$E_0$. Let $k \subset E_0$ be a subfield and let $\Sigma$ be a subset of~$\Emb(k)$. Then by a partial CM type relative to $(k,\Sigma)$ we mean a subset $\Phi \subset \Emb(E)$ such that the map $\phi \mapsto \phi|_{E_0}$ gives a bijection 
\[
\Phi \isomarrow \bigl\{\phi_0 \in \Emb(E_0) \bigm| \phi_0|_k \in \Sigma \bigr\}\, .
\] 

We say that a partial CM type $\Phi \subset \Emb(E)$ relative to $(k,\Sigma)$ is primitive if for every $\phi \neq \phi^\prime$ in~$\Emb(E)$ with $\phi|_k = \phi^\prime|_k$, there exists an element $\gamma \in \Gamma_\bbQ$ such that $\gamma \circ \phi \in \Phi$ while $\gamma \circ \phi^\prime \notin \Phi$.
\end{definition}

Note that in other places in the literature (e.g., \cite{PST}, Section~3.1) the term `partial CM type' is used for any subset $\Phi \subset \Emb(E)$ with $\Phi \cap \overline\Phi = \emptyset$. Of course, any such~$\Phi$ is a partial CM type in our sense for some choice of~$(k,\Sigma)$, as we can just take $k = E_0$ and $\Sigma = \bigl\{\phi|_{E_0}\bigm| \phi \in \Phi\bigr\}$; but the condition for $\Phi$ to be primitive depends on the choice of~$(k,\Sigma)$, see the next examples.

\subsection{Examples.}
(1) If we take $k = \bbQ$ and $\Sigma = \Emb(\bbQ)$, we recover the usual notion of a CM type. As we shall show, such a CM type~$\Phi$ is primitive in the above sense if and only if it is primitive in the classical sense, i.e., if $\Phi$ is not induced from a proper CM subfield of~$E$. See Remark~\ref{rem:PartialPhi}.

(2) Suppose we take $k=E_0$. Then a partial CM type $\Phi \subset \Emb(E)$ relative to $(k,\Sigma)$ is primitive whenever $\Sigma \neq \emptyset$. Indeed, if $k = E_0$ and $\Sigma \neq \emptyset$ then for any $\phi \neq \phi^\prime$ in~$\Emb(E)$ with $\phi|_{E_0} = \phi_0 = \phi^\prime|_{E_0}$, we can find an element $\gamma \in \Gamma_\bbQ$ such that $\gamma \circ \phi_0 \in \Sigma$. Then $\gamma \circ \phi$ and $\gamma \circ \phi^\prime$ are the only two elements of~$\Emb(E)$ that restrict to~$\gamma \circ \phi_0$, so precisely one of them lies in~$\Phi$. Possibly after composing~$\gamma$ with complex conjugation we find that $\gamma \circ \phi \in \Phi$ and $\gamma \circ \phi^\prime \notin \Phi$.

(3) In the situation considered in Proposition~\ref{prop:Albert4}, we have a partial CM type~$\Phi$ relative to $k = k_{\biI}$ and $\Sigma = \bigl\{I \in \biI \bigm| \sigma_{\nc} \notin I\bigr\}$. (Here we identify $\Emb(k_{\biI})$ with~$\biI$.)

\begin{proposition}\label{prop:PhiProp}
Let notation and assumptions be as in Proposition~\ref{prop:Albert4}. Then the partial CM type~$\Phi$ relative to $\bigl(k_{\biI},\{I \in \biI | \sigma_{\nc} \notin I\}\bigr)$ is primitive.
\end{proposition}

\begin{proof}
Recall from~\ref{exa:TK} that $\St_E$ denotes the standard representation of~$T_E$. The character group~$\mathsf{X}^*(Z)$ of~$Z$ is a quotient of~$\mathsf{X}^*(T_E) = \oplus_{\phi \in \Emb(E)}\; \bbZ\cdot \mathrm{e}_\phi$. Write $\bar{\mathrm{e}}_\phi \in \mathsf{X}^*(Z)$ for the image of~$\mathrm{e}_\phi$. Because $\Gamma_\bbQ$ acts transitively on the set of elements~$\mathrm{e}_\phi$, the index $\nu = \bigl[\Stab(\bar{\mathrm{e}}_\phi):\Stab(\mathrm{e}_\phi)\bigr]$ is independent of~$\phi$. Let $K$ be the centre of $\End_Z(\St_E)$; then $K$ is a subfield of~$E$ with $[E:K] = \nu$. Further, $\End_Z(\St_E) \cong M_\nu(K)$, and the embedding $E = \End_{T_E}(\St_E) \hookrightarrow \End_Z(\St_E)$ realizes $E$ as a maximal commutative subalgebra of~$\End_Z(\St_E)$.

Let $Kk_{\biI} \subset E$ be the compositum of the subfields~$K$ and~$k_{\biI}$. We first show that the equality $\End_{Z\times G^\sconn}(V) = \End_{T_E \times G^\sconn}(V)$ holds if and only if $Kk_{\biI} = E$. The isomorphism \eqref{eq:EndGscV} restricts to an isomorphism
\begin{equation}\label{eq:EndZGscV}
\End_{Z\times G^\sconn}(V) \cong \End_{\cE_{\biI}}(H) \cap \End_Z(H)\, .
\end{equation}
In either of the cases distinguished in Proposition~\ref{prop:Albert4}, $H$ is isomorphic, as a representation of~$T_E$, to a direct sum of copies of the standard representation~$\St_E$. The inclusion $K \hookrightarrow E$ gives~$H$ the structure of a $K$-vector space, and the relation $\End_Z(\St_E) \cong M_\nu(K)$ implies that $\End_Z(H) = \End_K(H)$. Therefore, the right hand side of~\eqref{eq:EndZGscV} equals $\End_{K\cE_{\biI}}(H)$, where $K\cE_{\biI}$ is the subalgebra of $E\cE_{\biI}$ generated by $K$ and~$\cE_{\biI}$. Recall that $\cE_{\biI}$ is a quaternion algebra over~$k_{\biI}$, so $K\cE_{\biI} \cong Kk_{\biI} \otimes_{k_{\biI}} \cE_{\biI}$ is a $4$-dimensional central simple algebra over~$Kk_{\biI}$. On the other hand, $\End_{T_E \times G^\sconn}(V) = \End_{E\cE_{\biI}}(H)$. By the double centraliser theorem, we have $\End_{Z\times G^\sconn}(V) = \End_{T_E \times G^\sconn}(V)$ if and only if $K\cE_{\biI} = E\cE_{\biI}$. Because $E\cE_{\biI}$ is a $4$-dimensional central simple algebra over~$E$, this is in turn equivalent to $Kk_{\biI} = E$. 

Next we show that we have an equality $Kk_{\biI} = E$ if and only if $\Phi$ is primitive. For $\phi \in \Emb(E)$, write $P(\phi) = \bigl\{\phi^\prime \in \Emb(E) \bigm|  \phi|_{k_{\biI}} = \phi^\prime|_{k_{\biI}}\bigr\}$, which is the fibre of the restriction map $\Emb(E) \to \Emb(k_{\biI})$ that contains~$\phi$. Then
\[
\Gal\bigl(\Qbar/\phi(E)\bigr) = \Stab(\phi) = \Stab(\mathrm{e}_\phi) \subset \Gal\bigl(\Qbar/\phi(k_{\biI})\bigr) = \Stab\bigl(P(\phi)\bigr)\, .
\]
(By $\Stab\bigl(P(\phi)\bigr) \subset \Gamma_\bbQ$ we mean the stabilizer of~$P(\phi)$ as a set, not the pointwise stabilizer. Note that the $\Gamma_\bbQ$-action permutes the sets of the form~$P(\phi)$, so that indeed $\Stab(\phi) \subset \Stab\bigl(P(\phi)\bigr)$.) On the other hand, it follows from Proposition~\ref{prop:repdescr}\ref{repdescr1} together with the characterisation of~$Z$ as the smallest subtorus of~$T_E$ such that $\mu_{\mathrm{c}}$ factors through~$Z_\bbC$, that
\[
\Stab(\bar{\mathrm{e}}_\phi) = \bigl\{\delta \in \Gamma_\bbQ \bigm| \ff(\gamma \circ \delta \circ \phi) = \ff(\gamma \circ \phi) \quad \text{for all $\gamma \in \Gamma_\bbQ$} \bigr\}\, .
\]
The equality $Kk_{\biI} = E$ is equivalent to the condition that for some (equivalently: for every) $\phi \in \Emb(E)$, the inclusion $\Stab(\phi) \subset \Stab\bigl(P(\phi)\bigr) \cap \Stab(\bar{\mathrm{e}}_\phi)$ is an equality. In other words,
\begin{align*}
Kk_{\biI} \subsetneq E &\iff \Stab(\phi) \subsetneq \Stab\bigl(P(\phi)\bigr) \cap \Stab(\bar{\mathrm{e}}_\phi)\\
&\iff \text{$\exists\; \phi^\prime = \delta\circ \phi \in P(\phi)$ such that $\phi^\prime \neq \phi$ and $\ff(\gamma \circ \phi^\prime) = \ff(\gamma \circ \phi)$ for all $\gamma \in \Gamma_\bbQ$.}
\end{align*}
This proves the proposition.
\end{proof}

\subsection{Remark.}\label{rem:PartialPhi} 
Suppose we take $k = \bbQ$ and $\Sigma = \Emb(\bbQ)$. In this case, $\Phi$ is a CM type on~$E$ in the usual sense, and we claim that it is primitive in the sense of Definition~\ref{def:PartialCM} if and only if $\Phi$ is not induced from a proper CM subfield of~$E$. To see this, let $\mu \colon \bbG_{\mult,\bbC} \to T_{E,\bbC}$ be the cocharacter defined by~$\Phi$, i.e., $\mu = \sum_{\phi \in \Phi}\; \check{\mathrm{e}}_\phi$, and let $Z \subset T_E$ be the smallest subtorus such that $\mu$ factors through~$Z_\bbC$. If $X$ is a complex abelian variety of CM type $(E,\Phi)$ (which is uniquely determined up to isogeny), $Z$ is the Mumford--Tate group of~$X$. With notation as in the proof of Proposition~\ref{prop:PhiProp}, $\Phi$ is primitive if and only if $\Stab(\bar{\mathrm{e}}_\phi) = \Stab(\mathrm{e}_\phi) = \Stab(\phi)$, which is equivalent to the condition that $\End_Z(\St_E) = E$, which in turn is equivalent to the condition that $\End^0(X) = E$. It is classical that this happens if and only if $\Phi$ is not induced from a proper CM subfield of~$E$.

\section{A classification result}\label{sec:Classif}

\subsection{}
Let $g$ be a positive integer. If $(V,\varphi)$ is a symplectic space of dimension~$2g$ over~$\bbQ$, let $\fH(V,\phi)$ denote the space of homomorphisms $h \colon \bbS \to \GSp(V,\phi)_\bbR$ that define a Hodge structure of type $(-1,0) + (0,-1)$ on~$V$ for which $\pm (2\pi i)\cdot \varphi$ is a polarisation. The group $\GSp(V,\phi)\bigl(\bbR\bigr)$ acts transitively on~$\fH(V,\phi)$ and the pair $\bigl(\GSp(V,\phi),\fH(V,\phi)\big)$ is a Shimura datum, which is usually called a Siegel modular Shimura datum. To simplify notation, we denote this datum by~$\fS(V,\phi)$.

The associated tower of Shimura varieties $\Sh_K\bigl(\fS(V,\phi)\bigr)$, for $K$ running through the set of compact open subgroups of $\GSp_{2g}(V,\phi)\bigl(\Af\bigr)$, is isomorphic to the tower~$\cA_{g,K}$ of moduli spaces of $g$-dimensional principally polarized abelian varieties with a level~$K$ structure; see for instance \cite{DelTravShim},~\S~4.

If $(V_1,\phi_1)$ and $(V_2,\phi_2)$ are symplectic spaces and $f \colon (V_1,\phi_1) \to (V_2,\phi_2)$ is a similitude, $f$ induces an isomorphism of Shimura data $\fS(V_1,\phi_1) \isomarrow \fS(V_2,\phi_2)$. Conversely, every such isomorphism is induced by a similitude. In particular, all automorphisms of the Shimura datum $\fS(V,\phi)$ are inner, i.e., are given by conjugation with an element of~$\GSp(V,\phi)\bigl(\bbQ\bigr)$. (Note that $\GSp(V,\phi)$ has non-inner automorphisms, but these do not map $\fH(V,\phi)$ into itself.)

\subsection{}\label{subsec:(GD,YD)}
Let $F$ be a totally real number field, $D$ a $4$-dimensional central simple $F$-algebra, and assume there exists a unique embedding $\sigma_{\nc} \in \Emb(F)$ such that $D \otimes_{F,\sigma_{\nc}} \bbR \cong M_2(\bbR)$. Define $\cG_D = \Res_{F/\bbQ}\; \PGL_{1,D}$, and let $\cY_D$ be the $\cG_D(\bbR)$-conjugacy class of the homomorphism $\bbS \to \cG_{D,\bbR}$ that on the unique noncompact factor~$\PGL_2$ of~$\cG_{D,\bbR}$ is given by $a+bi \mapsto  \left[\begin{smallmatrix} a&-b \\ b&a\end{smallmatrix}\right]$ and that is trivial on the compact factors. The pair $(\cG_D,\cY_D)$ is a $1$-dimensional adjoint Shimura datum. Because $\cG_D$ is a $\bbQ$-simple group, it is the generic Mumford--Tate group on~$\cY_D$.

We claim that all automorphisms of the Shimura datum $(\cG_D,\cY_D)$ are inner, i.e., $D^*/F^* \isomarrow \Aut(\cG_D,\cY_D)$. The automorphism group of~$\cG_D$ is the group of pairs $(\alpha,f)$ where $\alpha \in \Aut(F)$ and $f \colon \PGL_{1,D} \isomarrow \PGL_{1,\alpha^*D}$ is an isomorphism of groups over~$F$. (See \cite{CGP}, Proposition~A.5.14; the result is stated there for simply connected groups but the same argument works for adjoint groups.) As $\alpha$ has to preserve~$\sigma_{\nc}$, only $\alpha = \id$ is possible. Since all automorphisms of~$\PGL_{1,D}$ are inner, this gives the claim.

\subsection{}\label{subsec:ClassifProb}
Our goal is to classify the $1$-dimensional Shimura subvarieties of~$\cA_g$. This leads us to consider triples $(G,\cY,\rho)$, where $(G,\cY)$ is a Shimura datum with $\dim(\cY) = 1$ such that $G$ is the generic Mumford--Tate group on~$\cY$, and $\rho \colon (G,\cY) \hookrightarrow \fS(V,\phi)$ is an embedding of $(G,\cY)$ into a Siegel modular Shimura datum. 

If we have two such triples $\bigl(G_i,\cY_i,\rho_i \colon (G_i,\cY_i) \hookrightarrow \fS(V_i,\phi_i)\bigr)$, for $i=1,2$, we say these are equivalent if there exist isomorphisms of Shimura data $\alpha \colon (G_1,\cY_1) \isomarrow (G_2,\cY_2)$ and $\beta\colon \fS(V_1,\phi_1) \isomarrow \fS(V_2,\phi_2)$ with $\beta \circ \rho_1 = \rho_2 \circ \alpha$. 

If $(G,\cY,\rho)$ is a triple as above, the adjoint Shimura datum of $(G,\cY)$ is of the form described in~\ref{subsec:(GD,YD)}. In what follows, we fix $F$ and~$D$ as in~\ref{subsec:(GD,YD)} and study only triples $(G,\cY,\rho)$ such that $(G^\ad,\cY^\ad) \cong (\cG_D,\cY_D)$. Let $m = [F:\bbQ]$, and continue to write $\sigma_{\nc} \in \Emb(F)$ for the unique embedding such that $D \otimes_{F,\sigma_{\nc}} \bbR \cong M_2(\bbR)$. We write $\cG_D^\sconn = \Res_{F/\bbQ}\; \SL_{1,D}$ for the simply connected cover of~$\cG_D$, and we fix identifications as in~\eqref{eq:GscExplicit} (with $\cG_D$ in place of $G$). If $\biI$ is a $\Gamma_\bbQ$-orbit of nonempty subsets of~$\Emb(F)$, we denote by $\rho_{\biI} \colon \cG_D^\sconn \to \GL(W_{\biI})$ the corresponding irreducible representation.

\subsection{}\label{subsec:PolExists}
The following summarizes what is explained in \cite{DelCorvalis}, Section~1.1. Let $(G,\cY)$ be a Shimura datum whose weight is defined over~$\bbQ$, and such that for some (equivalently: all) $h \in \cY$ the involution $\mathrm{Inn}\bigl(h(i)\bigr)$ of $(G/w(\bbG_{\mult}))_\bbR$ is a Cartan involution. 

Let $h_0 \in \cY$, and let $\rho \colon G \to \GL(V)$ be a representation such that $\rho \circ h_0$ defines a Hodge structure of weight~$n$ on~$V$, for some $n \in \bbZ$. Then there exist:
\begin{itemize}
\item a character $\nu \colon G \to \bbG_{\mult}$ such that $\nu \circ w \colon  \bbG_{\mult} \to \bbG_{\mult}$ is given by $z \mapsto z^{-2n}$,
\item a bilinear form $\phi \colon V \times V \to \bbQ(-n)$, 
\end{itemize}
such that $\phi(gv,gv^\prime) = \nu(g) \cdot \phi(v,v^\prime)$ for all $g \in G$ and $v$, $v^\prime \in V$, and such that $\phi$ is a polarization of the Hodge structure on~$V$ given by $\rho \circ h_0$. The form~$\phi$ is symmetric (resp.\ symplectic) if the weight~$n$ is even (resp.\ odd). If $h = g\cdot h_0 \in \cY$ is any other element, either~$\phi$ or~$-\phi$ (depending on the sign of~$\nu(g)$) is a polarization of the Hodge structure defined by $\rho \circ h$. 

For our purposes, it suffices to consider the case where $\cY$ is $1$-dimensional and the Hodge structure on~$V$ is of type $(-1,0) + (0,-1)$. {}From now on, we assume this. Let $h_0$, $\nu$ and~$\phi$ be as above. Then $\rho$ defines an embedding of~$(G,\cY)$ into the Siegel modular datum~$\fS(V,\phi)$ and $\bigl(G,\cY,\rho \colon (G,\cY) \hookrightarrow \fS(V,\phi)\bigr)$ is a triple as in~\ref{subsec:ClassifProb}. 

For a given representation~$\rho$, the form~$\phi$ is not unique in general. To analyse this, we first make the simplifying assumption that~$\rho$ is isotypical, i.e., $\rho$ is isomorphic to a sum of copies of an irreducible representation. (We shall return to the general case in Section~\ref{sec:Nonsimple}.) The endomorphism algebra $A = \End(\rho)$ is then a matrix algebra over a finite dimensional division algebra over~$\bbQ$. The involution~$\dagger$ on~$A$ that is induced by~$\phi$ is positive, so that $(A,\dagger)$ is a pair of the type classified by Albert. (See \cite{MAV}, \S~21, for instance.) The set $A^\sym = \bigl\{f\in A \bigm| f = f^\dagger\bigr\}$ of symmetric elements in~$A$ is a formally real Jordan algebra for the product given by $f_1 \star f_2 = (f_1f_2 + f_2f_1)/2$, and the totally positive elements in~$A^\sym$ form an open cone~$A^{\sym,+}$. (Note that $A^\sym$ is a Jordan algebra over~$\bbQ$, so by `cone' we here mean a cone in a $\bbQ$-vector space. Further, writing~$E_0$ for the field of $\dagger$-symmetric elements in the centre of~$A$, which is a totally real field, we call an element $f \in A^\sym$ totally positive if its image in the Jordan algebra $(A\otimes_{E_0,\iota} \bbR)^\sym$ is positive for every embedding $\iota \colon E_0 \to \bbR$.) With this notation, if $h_0 \in \cY$ is Hodge generic, every other polarization form for the Hodge structure given by $\rho\circ h_0$ is of the form $\psi(v,v^\prime) = \phi(av,v^\prime)$ for some $a \in A^{\sym,+}$. Conversely, for every such~$a$ the form $\psi_a(v,v^\prime) = \phi(av,v^\prime)$ is a polarization, and $\bigl(G,\cY,\rho \colon (G,\cY) \hookrightarrow \fS(V,\psi_a)\bigr)$ is again a triple as in~\ref{subsec:ClassifProb}. 

In general, it is somewhat complicated to say under what conditions on $a\in A^{\sym,+}$ the forms~$\phi$ and~$\psi_a$ give rise to equivalent triples. A sufficient condition for these triples to be equivalent is that there exists a $G$-equivariant similitude $(V,\phi) \to (V,\psi_a)$, which happens if and only if $a = \nu\cdot b^\dagger b$ for some $b \in A^*$ and $\nu\in \bbQ^*$. (Note that a Siegel modular datum~$\fS(V,\phi)$ does not change if we multiply~$\phi$ by a scalar in~$\bbQ^*$.) If all automorphisms of the Shimura datum $(G,\cY)$ are inner, this is also a necessary condition; but in general $(G,\cY)$ can have non-inner automorphisms.
\medskip

After these preliminaries, we now discuss two constructions that, together, describe all possible triples $(G,\cY,\rho)$ as in~\ref{subsec:ClassifProb} for which the representation~$\rho$ is isotypical. (See Theorem~\ref{thm:MainThm}.) This is essentially Section~\ref{sec:AV1dimSV} in reverse. We fix $(\cG_D,\cY_D)$ as in~\ref{subsec:ClassifProb}.

\subsection{Construction~1.}\label{Constr1} Let $\tilde{\rho}_{\mathrm{s}} \colon \cG_D^\sconn \to \GL(V)$ be the corestriction representation; this is the case $\biI = \{\Emb(F)\}$ that was discussed in Example~\ref{exa:111Rep}, applied with $L = F$. Let $\tilde{G} = \bbG_{\mult,\bbQ} \times \cG_D^\sconn$, and let $\tilde\rho \colon \tilde{G} \to \GL(V)$ be given by $(c,g) \mapsto c\cdot \tilde{\rho}_{\mathrm{s}}(g)$. Let $G = \Image(\tilde{\rho})$ be the image of~$\tilde\rho$, and let $\rho \colon G \hookrightarrow \GL(V)$ be the induced representation. 

Let $\tilde\mu \colon \bbG_{\mult,\bbC} \to \tilde{G}_\bbC$ be the cocharacter given by $z \mapsto z$ on the first factor of~$\tilde{G}$ and by formula~\eqref{eq:tildemus} on the factor~$\cG^\sconn_{D,\bbC}$. As $\bm{\mu}_2 = \bigl\{\pm 1\bigr\} \subset \bbG_{\mult,\bbC}$ lies in the kernel of~$\tilde\rho \circ \tilde\mu$, there exists a unique cocharacter $\mu \colon \bbG_{\mult,\bbC} \to G_\bbC$ such that $\tilde\rho \circ \tilde\mu$ lifts the square of $\rho \circ \mu$. Let $h \colon \bbG_{\mult,\bbR} \to G_\bbR$ be the homomorphism such that $h_\bbC = \mu \cdot \bar\mu$, and let $\cY$ be the $G(\bbR)$-conjugacy class of~$h$. The pair $(G,\cY)$ is a Shimura datum whose weight is defined over~$\bbQ$ such that $(G^\ad,\cY^\ad) \cong (\cG_D,\cY_D)$. (See also~\ref{thm:MainThm}.) It is clear from the construction that $G/w(\bbG_{\mult}) = G^\ad$ and that $G$ is the generic Mumford--Tate group on~$\cY$. 

Let $r$ be a positive integer, and consider the representation~$\rho^{\oplus r}$. 
By what was explained in~\ref{subsec:PolExists}, there exists a polarization form~$\phi$ on~$V^{\oplus r}$ such that $\rho^{\oplus r}$ factors through $\GSp(V^{\oplus r},\phi)$; we choose one. Then $\bigl(G,\cY,\rho \colon (G,\cY) \hookrightarrow \fS(V^{\oplus r},\phi)\bigr)$ is a triple as in~\ref{subsec:ClassifProb}.

The endomorphism algebra~$A_0$ of the representation~$\rho$ is described as in Proposition~\ref{prop:Albert123}\ref{Albert123-3}. The endomorphism algebra~$A$ of~$\rho^{\oplus r}$ is the matrix algebra~$M_r(A_0)$. Let $\dagger$ be the involution of~$A$ given by the chosen form~$\phi$, and let the notation $A^{\sym,+} \subset A^\sym$ be as in~\ref{subsec:PolExists}. For $a \in A^{\sym,+}$ the form~$\psi_a$ given by $\psi_a(v,v^\prime) = \phi(av,v^\prime)$ has the property that $\rho$ factors through $\GSp(V^{\oplus r},\psi_a)$, and $\bigl(G,\cY,\rho \colon (G,\cY) \hookrightarrow \fS(V^{\oplus r},\psi_a)\bigr)$ is again a triple as in~\ref{subsec:ClassifProb}.

\subsection{Construction~2.}\label{Constr2} Let $\biI$ be a $\Gamma_\bbQ$-orbit of nonempty subsets of~$\Emb(F)$, with $\biI \neq \bigl\{\Emb(F)\bigr\}$, and let $\rho_{\biI} \colon \cG_D^\sconn \to \GL(W_{\biI})$ be the corresponding irreducible representation. Let $\ell(\biI)$ be the cardinality of the sets in~$\biI$, let $\cE_{\biI} = \End(\rho_{\biI})$, and let $k_{\biI}$ be the centre of~$\cE_{\biI}$, which is a totally real field. As explained in Remark~\ref{rem:rhoI}, there is a natural identification $\biI = \Emb(k_{\biI})$, and we use this to view $\Sigma = \bigl\{I\in \biI\bigm| \sigma_{\nc} \notin I\bigr\}$ as a subset of~$\Emb(k_{\biI})$. Let $k_{\biI} \subset E_0$ be a finite totally real field extension and $E_0 \subset E$ a totally imaginary quadratic extension (so $E$ is a CM field). Let $\Phi \subset \Emb(E)$ be a primitive partial CM type relative to $(k_{\biI},\Sigma)$. (Note that such a type~$\Phi$ may not exist for all choices of~$E$.) Let $H$ be the unique simple right $(E \otimes_{k_{\biI}} \cE_{\biI})$-module, and define $V = H \otimes_{\cE_{\biI}} W_{\biI}$, viewed as a $\bbQ$-vector space.

Define a multiplication type~$\ff$ as in \eqref{eq:multtype}. Let $\tilde{\mu}_{\mathrm{c}} \colon \bbG_{\mult,\bbC} \to T_{E,\bbC}$ be the cocharacter that is given, as element of $\mathsf{X}_*(T_E)$, by $\tilde{\mu}_{\mathrm{c}} = \sum_{\phi \in \Emb(E)}\;  \frac{2\cdot \ff(\phi)}{n} \cdot \check{\mathrm{e}}_\phi$, where $n = \dim_E(V)$, which equals~$2^{\ell(\biI)}$ if the class of $E \otimes_{k_{\biI}} \cE_{\biI}$ in~$\Br(E)$ is trivial, and else equals $2^{\ell(\biI)+1}$. Let $Z \subset T_E$ be the smallest subtorus such that $\tilde{\mu}_{\mathrm{c}}$ factors through~$Z_\bbC$. Let $\tilde{\mu}_{\mathrm{s}} \colon \bbG_{\mult,\bbC} \to \cG_{D,\bbC}^\sconn$ be as in Proposition~\ref{prop:repdescr}\ref{repdescr2} (here with $G= \cG_D$), and let $\tilde{\mu} = (\tilde{\mu}_{\mathrm{c}},\tilde{\mu}_{\mathrm{s}}) \colon \bbG_{\mult,\bbC} \to Z_\bbC \times \cG_{D,\bbC}$.

We have a natural action of $T_E$ on~$H$ by $\cE_{\biI}$-linear automorphisms. This gives a representation of~$T_E$ on~$V$ which commutes with the action of~$\cG_D^\sconn$ and therefore defines a representation $\tilde\rho \colon T_E \times \cG_D^\sconn \to \GL(V)$. Let 
\[
G = \tilde{\rho}\bigl(Z\times \cG_D^\sconn \bigr) \subset \GL(V)
\]
be the image of $Z\times \cG_D^\sconn$ under~$\tilde\rho$, and write $\rho \colon G \hookrightarrow \GL(V)$ for the induced representation. 

We claim that $\bm{\mu}_2 \subset \bbG_{\mult,\bbC}$ lies in the kernel of~$\tilde\rho \circ \tilde\mu$. To see this, we observe that (depending on the structure of $E \otimes_{k_{\biI}} \cE_{\biI}$) either~$V$ or~$V^{\oplus 2}$ is isomorphic to $\St_E \boxtimes_{k_{\biI}} W_{\biI}$ as a representation of $T_E \times \cG_D^\sconn$. It therefore suffices to show that the action of $-1 \in \bbC^* = \bbG_{\mult}(\bbC)$ on $(\St_E \boxtimes_{k_{\biI}} W_{\biI}) \otimes_\bbQ \bbC$ is trivial. We have a decomposition~\eqref{eq:StEWItensorC}.  For a given $\phi \in \Emb(E)$, if $\sigma_{\nc} \in I_\phi$ then $\tilde\mu_{\mathrm{c}}(-1)$ and $\tilde\mu_{\mathrm{s}}(-1)$ both act on the summand $\boxtimes_{\sigma \in I_\phi}\; \St_\sigma$ as $-\id$; if $\sigma_{\nc} \notin I_\phi$ then both elements acts as the identity. In either case, therefore, we see that $-1 \in \bbC^* = \bbG_{\mult}(\bbC)$ acts as the identity, which proves our claim. It follows that there exists a unique cocharacter $\mu \colon \bbG_{\mult,\bbC} \to G_\bbC$ such that $\tilde\rho \circ \tilde\mu$ lifts the square of $\rho \circ \mu$. Let $h \colon \bbG_{\mult,\bbR} \to G_\bbR$ be the homomorphism such that $h_\bbC = \mu \cdot \bar\mu$, and let $\cY$ be the $G(\bbR)$-conjugacy class of~$h$. The pair $(G,\cY)$ is a Shimura datum whose weight is defined over~$\bbQ$ (see~\ref{thm:MainThm}), and $\rho \circ w \colon \bbG_{\mult,\bbR} \to \GL(V)_\bbR$ is given by $z \mapsto z\cdot \id_V$. By construction, the adjoint Shimura datum $(G^\ad,\cY^\ad)$ is isomorphic to $(\cG_D,\cY_D)$.

Let $r$ be a positive integer, and consider the representation~$\rho^{\oplus r}$.
It follows from its definition that $Z$ is contained in the torus $U_E \subset T_E$ given by $U_E = \bigl\{x \in E^* \bigm| x\bar{x} \in \bbQ^*\bigr\}$. As $(U_E/\bbG_{\mult})_\bbR$ is compact, $\mathrm{Inn}\bigl(h(i)\bigr)$ is a Cartan involution of $(G/w(\bbG_{\mult}))_\bbR$. It therefore follows from what was explained in~\ref{subsec:PolExists} that there exists a polarization form~$\phi$ on~$V^{\oplus r}$ such that $\rho^{\oplus r}$ factors through $\GSp(V^{\oplus r},\phi)$; we choose one. 

By Lemma~\ref{lem:Eisquat}, the assumption that $\Phi$ is primitive implies that $\cE_{\biI}$ is a quaternion algebra over~$k_{\biI}$. The endomorphism algebra of~$\tilde\rho$ is isomorphic to $\End_{E \otimes_{k_{\biI}} \cE_{\biI}}(H)$. The proof of Proposition~\ref{prop:PhiProp} shows that $\End(\rho) = \End(\tilde\rho)$; hence the representation~$\rho$ is irreducible. More precisely, if we write $A_0 = \End(\rho)$ then either $E \otimes_{k_{\biI}} \cE_{\biI} \cong M_2(E)$, in which case $A_0 = E$, or $E \otimes_{k_{\biI}} \cE_{\biI}$ is a quaternion algebra over~$E$, in which case $A_0 = E \otimes_{k_{\biI}} \cE_{\biI}$. The endomorphism algebra~$A$ of~$\rho^{\oplus r}$ is the matrix algebra~$M_r(A_0)$, on which~$\phi$ induced an involution. For $a \in A^{\sym,+}$, let $\psi_a(v,v^\prime) = \phi(av,v^\prime)$; then $\rho^{\oplus r}$ factors through $\GSp(V^{\oplus r},\psi_a)$, and $\bigl(G,\cY,\rho \colon (G,\cY) \hookrightarrow \fS(V^{\oplus r},\psi_a)\bigr)$ is a triple as in~\ref{subsec:ClassifProb}.

\begin{theorem}\label{thm:MainThm}
The constructions in \ref{Constr1} and \ref{Constr2} give triples $(G,\cY,\rho)$ as in~\ref{subsec:ClassifProb} with the property that the representation~$\rho$ is isotypical. Every such triple is obtained in this way.
\end{theorem}

\begin{proof}
In either case, the construction yields a $\bbQ$-group~$G$, a homomorphism $h \colon \bbS \to G_\bbR$, and an embedding $\rho \colon G \hookrightarrow \GL(V)$ such that
\begin{itemize}
\item the Hodge structure on~$V$ defined by $\rho\circ h$ is of type $(-1,0) + (0,-1)$;
\item $G^\ad \cong \cG_D$, and $\ad \circ h \in \cY_D$.
\end{itemize}
With $\cY = G(\bbR)\cdot h$, these properties imply that $(G,\cY)$ is a Shimura datum whose weight is defined over~$\bbQ$, such that $(G^\ad,\cY^\ad) \cong (\cG_D,\cY_D)$.

In either construction, it is clear that the triples $(G,\cY,\rho)$ that we obtain satisfy the conditions of~\ref{subsec:ClassifProb}. In Construction~1 it is clear that the representation~$\rho$ is irreducible, and as explained in~\ref{Constr2}, in Construction~2 the irreducibility follows from the assumption that $\Phi$ is primitive.

The last assertion follows from the results in Section~\ref{sec:AV1dimSV}.
\end{proof}

\subsection{Remark.}
In the situation considered in~\ref{Constr1} (Construction~1), all automorphisms of $(G,\cY)$ are inner. To see this, we use that $G^\ad \cong G/\bbG_{\mult}$, so that the map $G(\bbQ) \to G^\ad(\bbQ)$ is surjective. As explained in~\ref{subsec:(GD,YD)}, all automorphisms of the adjoint Shimura datum are inner. Therefore, if $\alpha$ is an automorphism of $(G,\cY)$ then possibly after changing~$\alpha$ by an inner automorphism, we may assume $\alpha$ induces the identity on~$G^\der$. The only non-inner automorphism of~$G$ with this property is the one given by $z \mapsto z^{-1}$ on the centre~$\bbG_{\mult}$ and by the identity on~$G^\der$; but this automorphism does not give an automorphism of $(G,\cY)$ (it does not even preserve the weight).

As discussed at the end of Section~\ref{subsec:PolExists}, the fact that all automorphisms of $(G,\cY)$ are inner implies that the forms~$\phi$ and~$\psi_a$ give rise to equivalent triples if and only if $a = \nu \cdot b^\dagger b$ for some $\nu \in \bbQ^*$ and $b \in A^*$. Let us also note that if $r=1$ and the Albert type is I or~III (i.e., either $m = [F:\bbQ]$ is even or the Brauer class of $\Cores_{F/\bbQ}(D)$ is trivial) then in fact $\phi$ is, up to scalars, the unique symplectic form on~$V$ such that $\rho$ factors through $\GSp(V,\phi)$.

In Construction~2 (Albert type~IV), it is in general more complicated to say when different polarization forms give rise to equivalent triples, as in this case $(G,\cY)$ may have non-inner automorphisms. It is of course still true that $\phi$ and~$\psi_a$ give equivalent triples if $a = \nu \cdot b^\dagger b$ for some $\nu \in \bbQ^*$ and $b \in A^*$, but this may in general not be a necessary condition.

\subsection{Remark.}\label{rem:ViehZuo}
As mentioned in the introduction, it appears that at some places in the literature there are misconceptions about the classification of `Shimura curves'. As a concrete example, we explain why Theorem~0.7 of~\cite{ViehZuo} is not true. We briefly recall the setting. The authors start (op.\ cit., Example~0.6) by considering a quaternion algebra~$A$ over a number field~$F$ such that $A$ splits at precisely one real place. If $L \subset F$ is a subfield, there exists an embedding $j\colon \Cores_{F/L}(A) \hookrightarrow M_{2^\mu}(L)$ with $\mu = [F:L]$ if $\Cores_{F/L}(A)$ has trivial Brauer class and $\mu = [F:L]+1$ otherwise. Next the authors say there exists a complex Shimura curve~$Y^\prime$ such that $j$ gives rise to a local system~$\mathbb{V}_L$ of $L$-vector spaces on~$Y^\prime$ whose underlying $\bbQ$-local system~$\mathbb{X}_{A,L}$ is irreducible. The algebraic monodromy group of~$\mathbb{X}_{A,L}$ has $\cG_A = \Res_{F/\bbQ} \PGL_{1,A}$ as its adjoint group, so up to isomorphism, $A$ is determined by~$\mathbb{X}_{A,L}$. 

It is not so hard to relate the construction of~$Y^\prime$ and~$\mathbb{X}_{A,L}$ to our classification. Let us do this in the cases $L=\bbQ$ and $L=F$. For $L=\bbQ$, the Shimura curve~$Y^\prime$ that is constructed in~\cite{ViehZuo} corresponds to the Shimura datum of our Construction~1 (see~\ref{Constr1}), and the local system~$\mathbb{X}_{A,\bbQ}$ is the one that corresponds to the embedding $\rho \colon (G,\cY) \hookrightarrow \fS(V,\phi)$ as in our construction; in other words, the monodromy representation is the corestriction representation as in Example~\ref{exa:111Rep}. If $L=F$, the Shimura curve that is obtained corresponds to the example we have discussed in~\ref{exa:unitary}; in this case the monodromy representation is as in Example~\ref{exa:rhoIrep}, where we take for~$\biI$ the set of singletons in~$\Emb(F)$. 

In \cite{ViehZuo}, Theorem~0.7, the authors consider a complete nonsingular curve~$Y$ and an abelian scheme $f \colon X \to Y$  that reaches the Arakelov bound. The assertion is that there exists a quaternion algebra~$A$ as above and an \'etale covering $Y^\prime \to Y$, for some~$Y^\prime$ as above, such that $X^\prime = X \times_Y Y^\prime$ decomposes, up to isogeny, as a product of a constant factor~$B$ and abelian schemes $h_i \colon Z_i \to Y^\prime$ ($i=1,\ldots,\ell$) whose generic fibres are simple, and such that each $R^1h_{i,*}\bbQ_{Z_i}$ is a direct sum of copies of~$\mathbb{X}_{A,L_i}$ for some subfield $L_i \subset F$.

According to \cite{Moeller}, Theorem~1.2, if $Y \hookrightarrow \cA_g$ is an irreducible component of a Shimura curve, the corresponding abelian scheme over~$Y$ reaches the Arakelov bound. Therefore, if we take an example as in Construction~2 (see~\ref{Constr2}) with $F$ of prime degree over~$\bbQ$ (so that $F$ has no subfields other than $\bbQ$ and~$F$) and with $\biI$ not the set of singletons in~$\Emb(F)$ (and of course also $\biI \neq \{\Emb(F)\}$), we obtain a counterexample to \cite{ViehZuo}, Theorem~0.7.

\section{The nonsimple case}\label{sec:Nonsimple}

\subsection{}\label{subsec:nonsimpleAV}
In Section~\ref{sec:AV1dimSV} we have discussed simple abelian varieties with $1$-dimensional associated Shimura datum. We now consider the general case. 

Let $X$ be a complex abelian variety. Let $(G,\cY)$ be the Shimura datum given by~$X$, and assume $\dim(\cY) = 1$. There exists an isogeny $f \colon X \to X_0 \times X_1^{m_1} \times \cdots \times X_t^{m_t}$, where $X_0$ is an abelian variety of CM type, $X_1,\ldots,X_t$ (with $t\geq 1$) are simple complex abelian varieties that are not of CM type, no two of which are isogenous, and $m_1,\ldots,m_t$ are positive integers. Write $V = H_1(X,\bbQ)$ and $V_i = H_1(X_i,\bbQ)$. The isogeny~$f$ induces an isomorphism $V \isomarrow V_0 \oplus V_1^{\oplus m_1} \oplus \cdots \oplus V_t^{\oplus m_t}$, and we use this to view $\prod_{i=0}^t\; \GL(V_i)$ as a subgroup of~$\GL(V)$, with $\GL(V_i)$ for $i \geq 1$ acting diagonally on~$V_i^{\oplus m_i}$. 

Let $(G_i,\cY_i)$ be the Shimura datum given by~$X_i$. Then $G = \MT(X) \subset \GL(V)$ is a subgroup of $G_0 \times G_1 \times \cdots \times G_t$. The projections $\pr_i \colon G \to G_i$ are surjective and for $i\geq 1$ they induce isomorphisms of adjoint Shimura data $\pr^\ad_i \colon (G^\ad,\cY^\ad) \isomarrow (G_i^\ad,\cY_i^\ad)$. The datum $(G,\cY)$ can be recovered from $(G^\ad,\cY^\ad)$ together with the data $(G_i,\cY_i)$ and the isomorphisms~$\pr^\ad_i$, as follows. We have
\begin{equation}\label{eq:fusingY}
\cY \xrightarrow[\prod p_i]{~\sim~} \Bigl\{(h_0,h_1,\ldots,h_t) \in \prod_{i=0}^t\; \cY_i \Bigm| (\pr^\ad_1)^{-1} \circ h_1 = \cdots = (\pr^\ad_t)^{-1} \circ h_t \Bigr\}\, , 
\end{equation}
where we recall that the adjoint map $G_i \to G_i^\ad$ identifies~$\cY_i$ with a union of connected components of~$\cY_i^\ad$. (Note that $\cY_0$ is a singleton.) The group~$G$ can be recovered as the smallest subgroup of~$\prod_{i=0}^t\, G_i$ such that all $h \in \cY$ factor through~$G_\bbR$.

\subsection{}
For the general classification of triples $(G,\cY,\rho)$ as in~\ref{subsec:ClassifProb}, the solution is obtained as follows. First we fix $F$ and~$D$ as in~\ref{subsec:(GD,YD)}, so that we have an adjoint datum $(\cG_D,\cY_D)$. Next we choose a triple $(G_0,\cY_0,\rho_0)$ where $(G_0,\cY_0)$ is a $0$-dimensional Shimura datum (so $G_0$ is a torus and $\cY_0$ is a singleton), and $\rho_0$ is an embedding $(G_0,\cY_0) \hookrightarrow \fS(V_0,\phi_0)$ into a Siegel modular datum. All such triples are described in terms of classical CM theory. As a next step, we fix finitely many triples $(G_i,\cY_i,\rho_i\colon (G_i,\cY_i) \hookrightarrow \fS(V_i,\phi_i))_{i=1,\ldots,t}$ as in~\ref{subsec:ClassifProb} such that $\rho_i$ is an isotypical representation, and such that there exist isomorphisms $p_i \colon (\cG_D,\cY_D) \isomarrow (G_i^\ad,\cY_i^\ad)$. For each~$i$ we fix such an isomorphism~$p_i$ (which, as remarked in~\ref{subsec:(GD,YD)}, is unique up to inner automorphisms of~$\cG_D$.) Via these choices we can view~$\rho_i \colon G_i \to \GL(V_i)$ as a representation of~$\cG_D^\sconn$. We make these choices in such a manner that there are no indices $1\leq i < j\leq t$ such that $\rho_i$ and~$\rho_j$, viewed as representations of~$\cG_D^\sconn$, have isomorphic underlying irreducible representations. This condition is independent of how we choose the isomorphisms~$p_i$. 

Define $V = V_0 \oplus V_1 \oplus \cdots \oplus V_t$ and define a symplectic form~$\phi$ on~$V$ by $\phi = \phi_0 \perp \phi_1 \perp \cdots \perp \phi_t$. This gives an embedding $\rho^\sharp \colon \prod_{i=0}^t\, G_i \hookrightarrow \GSp(V,\phi)$. Define~$\cY$ as in the right hand side of~\eqref{eq:fusingY} (with $p_i$ instead of~$\pr_i^\ad$), let $G \subset \prod_{i=0}^t\, G_i$ be the smallest $\bbQ$-subgroup such that all $h \in \cY$ factor through~$G_\bbR$, and let $\rho$ be the restriction of~$\rho^\sharp$ to~$G$. This gives us a triple $(G,\cY,\rho)$ as in~\ref{subsec:ClassifProb} and it follows from what was explained in~\ref{subsec:nonsimpleAV} that every such triple is obtained in this way.

%%
%% make bibliography small:
%%
%% end of small for the blibliography
\bigskip

\noindent
\texttt{b.moonen@science.ru.nl}

\noindent
Radboud University Nijmegen, IMAPP, Nijmegen, The Netherlands

%\newpage
%\input{Notes}


{\small

\begin{thebibliography}{99}
\setlength{\parskip}{0pt}
\setlength{\itemsep}{0pt plus 0.3ex}

\bibitem{Addington}
S.~Addington,
Equivariant holomorphic maps of symmetric domains.
Duke Math.\ J.\ 55 (1987), no.~1, 65--88.

\bibitem{Conrad}
B.~Conrad,
Reductive group schemes.
In: Autour des sch\'emas en groupes. Vol.~I,
Panor.\ Synth\`eses 42/43, Soc.\ Math.\ France, 2014, pp.\ 93--444. 

\bibitem{CGP}
B.~Conrad, O.~Gabber, G.~Prasad,
Pseudo-reductive groups (2nd ed.). 
New Math.\ Monographs 26, Cambridge Univ.\ Press, 2015. 

\bibitem{DelTravShim}
P.~Deligne,
Travaux de Shimura. 
S\'em.\ Bourbaki (1970/1971), Exp.\ No.~389. 
In: Lecture Notes in Math.\ 244, Springer, 1971, pp.\ 123--165. 

\bibitem{DelCorvalis}
P.~Deligne,
Vari\'et\'es de Shimura: interpr\'etation modulaire, et techniques de construction de mod\`eles canoniques.
In: Automorphic forms, representations and $L$-functions (Corvallis, 1977), Proc.\ Sympos.\ Pure Math.\ 33, AMS, 1979; Part 2, pp.\ 247--289.

\bibitem{Ferrand}
D.~Ferrand,
Un foncteur norme.
Bull.\ Soc.\ Math.\ France 126 (1998), no.~1, 1--49.

\bibitem{Moeller}
M.~M\"oller,
Shimura- and Teichm\"uller curves.
J.\ Modern Dynamics 5 (2011), no.~1, 1--32.

\bibitem{MumfordNote}
D.~Mumford,
A note of Shimura's paper ``Discontinuous groups and abelian varieties''.
Math.\ Ann.\ 181 (1969), 345--351.

\bibitem{MAV}
D.~Mumford, 
Abelian varieties.
Tata Inst.\ Fundam.\ Res.\ Stud.\ Math.\ 5,
Oxford University Press, 1970.

\bibitem{PST}
J.~Pila, A.~Shankar, J.~Tsimerman, 
Canonical heights on Shimura varieties and the Andr\'e--Oort conjecture
(with an appendix by H.~Esnault and M.~Groechenig).
Preprint 2021, arXiv:2109.08788.

\bibitem{Riehm}
C.~Riehm,
The corestriction of algebraic structures.
Invent.\ Math.\ 11 (1970), 73--98.

\bibitem{Satake}
I.~Satake,
Symplectic representations of algebraic groups satisfying a certain analyticity condition.
Acta Math.\ 117 (1967), 215--279.

\bibitem{Serre}
J-P.~Serre,
Groupes alg\'ebriques associ\'es aux modules de Hodge--Tate.
In: Journ\'ees de G\'eom\'etrie Alg\'ebrique de Rennes, Ast\'erisque 65, Soc.\ Math.\ France, 1979; Vol.~III, pp.\ 155--188.

\bibitem{Tankeev}
S.~Tankeev,
Algebraic cycles on abelian varieties. II.
Izv.\ Akad.\ Nauk SSSR Ser.\ Mat.\ 43 (1979), no.~2, 418--429. 
[English translation: Math.\ USSR-Izv.\ 14 (1980), no.~2, 383--394.]

\bibitem{Tits}
J.~Tits,
Repr\'esentations lin\'eaires irr\'eductibles d'un groupe r\'eductif sur un corps quelconque.
J.\ reine angew.\ Math.\ 247 (1971), 196--220. 

\bibitem{ViehZuo}
E.~Viehweg, K.~Zuo,
A characterization of certain Shimura curves in the moduli stack of abelian varieties. 
J.\ Diff.\ Geom.\ 66 (2004), 233--287.

\end{thebibliography}
}
\end{document}